\documentclass[reqno,a4paper]{amsart}

\usepackage{amsthm}
\usepackage{amsmath}
\usepackage[all]{xy} \SelectTips{eu}{}
\usepackage{latexsym,amsfonts,amssymb}
\usepackage{hyperref}
\usepackage{mathptmx}
\usepackage{nicefrac}
\usepackage{array}
\usepackage{xcolor}
\usepackage{mathrsfs}
\usepackage{rotating}

\renewcommand{\theequation}{$\smash{\sharp}\mspace{0.5mu}$\arabic{equation}}

\newcommand{\numberseries}{\mdseries}   %Fontseries used for numbering theorem

\newlength{\thmtopspace}                %Space above theorem
\newlength{\thmbotspace}                %Space below theorem
\newlength{\thmheadspace}               %Space between theorem caption and text
\newlength{\thmindent}                  %For indenting

\newtheoremstyle{bfupright head,slanted body}
                {\thmtopspace}{\thmbotspace}
                {\slshape}{\thmindent}{\bfseries}{.}{\thmheadspace}
                {{\numberseries \thmnumber{{\bf #2} }}\thmnote{#3}}

\newtheoremstyle{bfupright head,upright body}
                {\thmtopspace}{\thmbotspace}
                {\upshape}{\thmindent}{\bfseries}{.}{\thmheadspace}
                {{\numberseries \thmnumber{{\bf #2} }}\thmnote{#3}}

\newtheoremstyle{bfit head,upright body}
                {\thmtopspace}{\thmbotspace}
                {\upshape}{\thmindent}{\upshape}{.}{\thmheadspace}
                {{\numberseries\thmnumber{{\bf #2} }}
                {\bfseries\itshape\thmnote{\negthickspace#3}}}

\newtheoremstyle{it head,upright body}
                {\thmtopspace}{\thmbotspace}
                {\upshape}{\thmindent}{\upshape}{.}{\thmheadspace}
                {{\numberseries\thmnumber{{\bf #2} }}
                {\itshape\thmnote{\negthickspace#3}}}

\newtheoremstyle{fixed bf head,slanted body}
                {\thmtopspace}{\thmbotspace}{\slshape}
                {\thmindent}{\bfseries}{.}{\thmheadspace}
                {{\numberseries \thmnumber{{\bf #2} }}\thmname{#1}\thmnote{ (#3)}}

\newtheoremstyle{fixed bf head,upright body}
                {\thmtopspace}{\thmbotspace}{\upshape}
                {\thmindent}{\bfseries}{.}{\thmheadspace}
                {{\numberseries \thmnumber{{\bf #2} }}\thmname{#1}\thmnote{ (#3)}}

\newtheoremstyle{independent paragraph}
                {\thmtopspace}{\thmbotspace}
                {\upshape}{\thmindent}{\upshape}{}{0pt}
                {\thmnote{#3 }}

\newtheoremstyle{subparagraph}
                {\thmbotspace}{\thmbotspace}
                {\upshape}{\thmindent}{\upshape}{}{0pt}
                {\thmnote{#3 }}

\newtheoremstyle{notes}
                {\thmtopspace}{\thmbotspace}
                {\ttfamily}{\thmindent}{\ttfamily\small }{}{0pt}
                {\thmnote{#3 }}

\theoremstyle{bfupright head,slanted body}
\newtheorem{res}{}[section]             \newtheorem*{res*}{}

\theoremstyle{bfit head,upright body}
                 \newtheorem*{com*}{}

\theoremstyle{bfupright head,upright body}
\newtheorem{bfhpg}[res]{}               \newtheorem*{bfhpg*}{}

\theoremstyle{it head,upright body}
               \newtheorem*{ithpg*}{}

\theoremstyle{fixed bf head,slanted body}
\newtheorem{thm}[res]{Theorem}          \newtheorem*{thm*}{Theorem}
\newtheorem{prp}[res]{Proposition}      \newtheorem*{prp*}{Proposition}
\newtheorem{cor}[res]{Corollary}        \newtheorem*{cor*}{Corollary}
\newtheorem{lem}[res]{Lemma}            \newtheorem*{lem*}{Lemma}

\theoremstyle{fixed bf head,upright body}
\newtheorem{dfn}[res]{Definition}       \newtheorem*{dfn*}{Definition}
\newtheorem{obs}[res]{Observation}      \newtheorem*{obs*}{Observation}
\newtheorem{rmk}[res]{Remark}           \newtheorem*{rmk*}{Remark}
\newtheorem{exa}[res]{Example}          \newtheorem*{exa*}{Example}
         \newtheorem*{exe*}{Exercise}
\newtheorem{stp}[res]{Setup}            \newtheorem{stp*}{Setup}

\theoremstyle{independent paragraph}

\theoremstyle{subparagraph}

\theoremstyle{notes}

\newlength{\thmlistleft}        %leftmargin
\newlength{\thmlistright}       %rightmargin
\newlength{\thmlistpartopsep}   %partopsep
\newlength{\thmlisttopsep}      %topsep
\newlength{\thmlistparsep}      %parsep
\newlength{\thmlistitemsep}     %itemsep

\newcounter{eqc} 
\newenvironment{eqc}{\begin{list}{\upshape (\textit{\roman{eqc}})}%
    {\usecounter{eqc}%
      \setlength{\leftmargin}{\thmlistleft}%
      \setlength{\labelwidth}{\thmlistleft}%
      \setlength{\rightmargin}{\thmlistright}%
      \setlength{\partopsep}{\thmlistpartopsep}%
      \setlength{\topsep}{\thmlisttopsep}%
      \setlength{\parsep}{\thmlistparsep}%
      \setlength{\itemsep}{\thmlistitemsep}}}%
  {\end{list}}%

\newcommand{\eqclbl}[1]{{\upshape(\textit{#1})}}

\newcounter{prt}
\newenvironment{prt}{\begin{list}{\upshape (\alph{prt})}%
    {\usecounter{prt}%
      \setlength{\leftmargin}{\thmlistleft}%
      \setlength{\labelwidth}{\thmlistleft}%
      \setlength{\rightmargin}{\thmlistright}%
      \setlength{\partopsep}{\thmlistpartopsep}%
      \setlength{\topsep}{\thmlisttopsep}%
      \setlength{\parsep}{\thmlistparsep}%
      \setlength{\itemsep}{\thmlistitemsep}}}%
  {\end{list}}%

\newcommand{\prtlbl}[1]{{\upshape(#1)}}

\newcounter{rqm}
\newenvironment{rqm}{\begin{list}{\upshape (\arabic{rqm})}%
    {\usecounter{rqm}%
      \setlength{\leftmargin}{\thmlistleft}%
      \setlength{\labelwidth}{\thmlistleft}%
      \setlength{\rightmargin}{\thmlistright}%
      \setlength{\partopsep}{\thmlistpartopsep}%
      \setlength{\topsep}{\thmlisttopsep}%
      \setlength{\parsep}{\thmlistparsep}%
      \setlength{\itemsep}{\thmlistitemsep}}}%
  {\end{list}}%

\newcommand{\rqmlbl}[1]{{\upshape(#1)}}

  {\end{list}}%

\newenvironment{itemlist}{\nopagebreak \begin{list}{$\bullet$}%
    {\setlength{\leftmargin}{\thmlistleft}%
      \setlength{\labelwidth}{\thmlistleft}%
      \setlength{\rightmargin}{\thmlistright}%
      \setlength{\partopsep}{\thmlistpartopsep}%
      \setlength{\topsep}{\thmlisttopsep}%
      \setlength{\parsep}{\thmlistparsep}%
      \setlength{\itemsep}{\thmlistitemsep}}}%
  {\end{list}}%

  \newcommand{\proofoftag}[2][:]{(#2)#1}

\newcommand{\pgref}[1]{\ref{#1}}

\renewcommand{\eqref}[1]{(\pgref{eq:#1})}

\newcommand{\corref}[2][Corollary ]{#1\pgref{cor:#2}}
\newcommand{\dfnref}[2][Definition ]{#1\pgref{dfn:#2}}
\newcommand{\exaref}[2][Example ]{#1\pgref{exa:#2}}
\newcommand{\lemref}[2][Lemma ]{#1\pgref{lem:#2}}
\newcommand{\obsref}[2][Observation ]{#1\pgref{obs:#2}}
\newcommand{\prpref}[2][Proposition ]{#1\pgref{prp:#2}}

\newcommand{\rmkref}[2][Remark ]{#1\pgref{rmk:#2}}
\newcommand{\thmref}[2][Theorem ]{#1\pgref{thm:#2}}
\newcommand{\stpref}[2][Setup ]{#1\pgref{stp:#2}}
\newcommand{\secref}[2][Section ]{#1\pgref{sec:#2}}

\makeatletter
\def\@nobreak@#1{\mathchoice%
  {\nobreakdef@\displaystyle\f@size{#1}}%
  {\nobreakdef@\nobreakstyle\tf@size{\firstchoice@false #1}}%
  {\nobreakdef@\nobreakstyle\sf@size{\firstchoice@false #1}}%
  {\nobreakdef@\nobreakstyle\ssf@size{\firstchoice@false #1}}%
  \check@mathfonts}%
\def\nobreakdef@#1#2#3{\hbox{{%
                    \everymath{#1}%
                    \let\f@size#2\selectfont%
                    #3}}}%
\makeatother

\DeclareFontFamily{T1}{cmex}{}
\DeclareFontShape{T1}{cmex}{m}{n}{<-> s * [0.89] cmex10}{}
\DeclareSymbolFont{cmlargesymbols}{T1}{cmex}{m}{n}
\DeclareMathSymbol{\mycoprod}{\mathop}{cmlargesymbols}{"60} 
\DeclareMathSymbol{\myprod}{\mathop}{cmlargesymbols}{"51} \let\prod\myprod

\DeclareSymbolFont{usualmathcal}{OMS}{cmsy}{m}{n}
\DeclareSymbolFontAlphabet{\mathcal}{usualmathcal}

\DeclareSymbolFont{letters}{OML}{txmi}{m}{it}

\DeclareMathSymbol{\alpha}{\mathord}{letters}{"0B}
\DeclareMathSymbol{\beta}{\mathord}{letters}{"0C}
\DeclareMathSymbol{\gamma}{\mathord}{letters}{"0D}
\DeclareMathSymbol{\sigma}{\mathord}{letters}{"0E}
\DeclareMathSymbol{\epsilon}{\mathord}{letters}{"0F}
\DeclareMathSymbol{\zeta}{\mathord}{letters}{"10}
\DeclareMathSymbol{\eta}{\mathord}{letters}{"11}
\DeclareMathSymbol{\theta}{\mathord}{letters}{"12}
\DeclareMathSymbol{\iota}{\mathord}{letters}{"13}
\DeclareMathSymbol{\kappa}{\mathord}{letters}{"14}
\DeclareMathSymbol{\lambda}{\mathord}{letters}{"15}
\DeclareMathSymbol{\mu}{\mathord}{letters}{"16}
\DeclareMathSymbol{\nu}{\mathord}{letters}{"17}
\DeclareMathSymbol{\xi}{\mathord}{letters}{"18}
\DeclareMathSymbol{\pi}{\mathord}{letters}{"19}
\DeclareMathSymbol{\rho}{\mathord}{letters}{"1A}
\DeclareMathSymbol{\sigma}{\mathord}{letters}{"1B}
\DeclareMathSymbol{\tau}{\mathord}{letters}{"1C}
\DeclareMathSymbol{\upsilon}{\mathord}{letters}{"1D}
\DeclareMathSymbol{\phi}{\mathord}{letters}{"1E}
\DeclareMathSymbol{\chi}{\mathord}{letters}{"1F}
\DeclareMathSymbol{\psi}{\mathord}{letters}{"20}
\DeclareMathSymbol{\omega}{\mathord}{letters}{"21}
\DeclareMathSymbol{\varepsilon}{\mathord}{letters}{"22}
\DeclareMathSymbol{\vartheta}{\mathord}{letters}{"23}
\DeclareMathSymbol{\varpi}{\mathord}{letters}{"24}
\DeclareMathSymbol{\varrho}{\mathord}{letters}{"25}
\DeclareMathSymbol{\varsigma}{\mathord}{letters}{"26}
\DeclareMathSymbol{\varphi}{\mathord}{letters}{"27}

\DeclareMathSymbol{\Gamma}{\mathord}{letters}{"00}
\DeclareMathSymbol{\Delta}{\mathord}{letters}{"01}
\DeclareMathSymbol{\Theta}{\mathord}{letters}{"02}
\DeclareMathSymbol{\Lambda}{\mathord}{letters}{"03}
\DeclareMathSymbol{\Xi}{\mathord}{letters}{"04}
\DeclareMathSymbol{\Pi}{\mathord}{letters}{"05}
\DeclareMathSymbol{\Sigma}{\mathord}{letters}{"06}
\DeclareMathSymbol{\Upsilon}{\mathord}{letters}{"07}
\DeclareMathSymbol{\upPhi}{\mathord}{letters}{"08}
\DeclareMathSymbol{\upPsi}{\mathord}{letters}{"09}
\DeclareMathSymbol{\Omega}{\mathord}{letters}{"0A}

\DeclareMathSymbol{\upGamma}{\mathalpha}{operators}{"00}
\DeclareMathSymbol{\upDelta}{\mathalpha}{operators}{"01}
\DeclareMathSymbol{\upTheta}{\mathalpha}{operators}{"02}
\DeclareMathSymbol{\upLambda}{\mathalpha}{operators}{"03}
\DeclareMathSymbol{\upXi}{\mathalpha}{operators}{"04}
\DeclareMathSymbol{\upPi}{\mathalpha}{operators}{"05}
\DeclareMathSymbol{\upSigma}{\mathalpha}{operators}{"06}
\DeclareMathSymbol{\upUpsilon}{\mathalpha}{operators}{"07}
\DeclareMathSymbol{\upPhi}{\mathalpha}{operators}{"08}
\DeclareMathSymbol{\upPsi}{\mathalpha}{operators}{"09}
\DeclareMathSymbol{\upOmega}{\mathalpha}{operators}{"0A}

\DeclareMathAlphabet\PazoBB{U}{fplmbb}{m}{n}%

\newcommand{\V}{\mathcal{V}}
\newcommand{\A}{\mathcal{A}}
\newcommand{\Ab}{\mathsf{Ab}}
\newcommand{\Set}{\mathsf{Set}}
\newcommand{\Hom}{\mathrm{Hom}}
\newcommand{\Coker}[1]{\operatorname{Cok}#1}
\newcommand{\lCocon}[3][\lambda]{#1\text{-}\mathrm{Cocont}(#2,#3)}
\newcommand{\lCon}[3][\lambda]{#1\text{-}\mathrm{Cont}(#2,#3)}
\newcommand{\lFlat}[3][\lambda]{#1\text{-}\mathrm{Flat}(#2,#3)}
\newcommand{\lPres}[2][\lambda]{\mathrm{Pres}_{#1}(#2)}
\newcommand{\fp}[1]{\mathrm{fp}(#1)}
\newcommand{\PureInj}[1]{\mathrm{PureInj}_\otimes(#1)}
\newcommand{\Inj}[1]{\mathrm{Inj}(#1)}
\newcommand{\AbsPure}[1]{\mathrm{AbsPure}(#1)}

\newcommand{\colim}{\mathrm{colim}}

\newcommand{\Ch}[1]{\mathsf{Ch}(#1)}
\newcommand{\stalk}[1]{S\mspace{-2mu}(#1)}
\newcommand{\disc}[1]{D(#1)}
\newcommand{\tHom}{\mathrm{Hom}_R^{\scriptscriptstyle\bullet}}
\newcommand{\tTen}[1][R]{\otimes_{#1}^{\scriptscriptstyle\bullet}}
\newcommand{\mHom}{\underline{\mathrm{Hom}}_{\mspace{2mu}R}^{\mspace{1mu}\scriptscriptstyle\bullet}}
\newcommand{\mTen}[1][R]{\underline{\otimes}_{\mspace{2mu}#1}^{\mspace{1mu}\scriptscriptstyle\bullet}}

\newcommand{\Mod}[1]{\mathsf{Mod}(#1)}
\newcommand{\Qcoh}[1]{\mathsf{Qcoh}(#1)}
\newcommand{\Coh}[1]{\mathsf{Coh}(#1)}
\newcommand{\shHomqce}{\mathscr{H}\mspace{-2.5mu}om^\mathrm{qc}}
\newcommand{\shHome}{\mathscr{H}\mspace{-2.5mu}om}

\begin{document}

\title{The tensor embedding for a Grothendieck cosmos}

\author{Henrik Holm \ }
\address[H.H.]{Department of Mathematical Sciences, University of Co\-penhagen, Universitetsparken 5, 2100 Copenhagen {\O}, Denmark} 
\email{holm@math.ku.dk}
\urladdr{http://www.math.ku.dk/\~{}holm/}

\author{ \ Sinem Odaba\c{s}{\i}}
\address[S.O.]{Instituto de Ciencias F\'isicas y Matem\'aticas, Universidad Austral de Chile, Valdivia-CHILE} 
\email{sinem.odabasi@uach.cl}
\thanks{S. Odaba\c{s}{\i} has been supported by the research grant CONICYT/FONDECYT/Iniciaci\'on/11170394.}

%\urladdr{http://www.math.ku.dk/\ {}holm/}

%\thanks{This work was initiated in June 2015 when the first author visited the Department of Mathematical Sciences at the University of Copenhagen. We are grateful to the department for its hospitality.}

\keywords{Cosmos; enriched functor; exact category; Grothendick category; (pre)envelope; (pure) injective object; purity; symmetric monoidal category; tensor embedding; Yoneda embedding.}

\subjclass[2010]{18D15, 18D20, 18E10, 18E15, 18E20, 18G05}

%18Dxx  CATEGORIES WITH STRUCTURE 
%18D15  Closed categories (closed monoidal and Cartesian closed categories, etc.) 
%18D20  Enriched categories (over closed or monoidal categories) 

%18Exx  ABELIAN CATEGORIES 
%18E10  Exact categories, abelian categories 
%18E15  Grothendieck categories 
%18E20  Embedding theorems  

%18Gxx  HOMOLOGICAL ALGEBRA 
%18G05  Projectives and injectives  

\begin{abstract}
While the Yoneda embedding and its generalizations have been studied extensively in the literature, the so-called tensor embedding has only received little attention. In this paper, we study the tensor embedding for closed symmetric monoidal categories and show how it is connected to the notion of geometrically purity, which has recently~been investigated in works of Enochs, Estrada, Gillespie, and Odaba\c{s}{\i}. More precisely, for~a~Gro\-thendieck cosmos---that is, a bicomplete Grothendick category $\V$ with a closed symmetric monoidal structure---we prove that the geometrically pure exact category $(\V,\mathscr{E}_\otimes)$ has enough relative injectives; in fact, every object has a geometrically pure injective envelope. We also show that for some regular cardinal $\lambda$, the tensor embedding yields an exact equivalence between $(\V,\mathscr{E}_\otimes)$ and the category of $\lambda$-cocontinuous $\V$-functors from $\lPres{\V}$ to $\V$, where the former is the full $\V$-subcategory of $\lambda$-presentable objects in $\V$. In many cases of interest, $\lambda$ can be chosen to be $\aleph_0$ and the tensor embedding identifies the geometrically pure injective objects in $\V$ with the (categorically) injective objects in the abelian category of $\V$-functors from $\fp{\V}$ to $\V$. As we explain, the developed theory applies  e.g.~to the category $\Ch{R}$ of chain complexes of modules over a commutative ring $R$ and to the category $\Qcoh{X}$ of quasi-coherent sheaves over a (suitably nice) scheme $X$.
\end{abstract}

\maketitle

\section{Introduction}

By the Gabriel--Quillen Embedding Theorem, see \cite[Thm.~A.7.1]{TT90}, any small exact category admits an exact full embedding, which also reflects exactness, into some abelian category. Hence any small exact category is equivalent, as an exact category, to an extension-closed sub\-category of an abelian category. Actually, the same is true for many large exact categories of interest. Consider e.g.~the category $R\textnormal{-Mod}$ of left $R$-modules equipped with the \emph{pure exact structure}, $\mathscr{E}_\mathrm{pure}$, where the ``exact sequences'' (the conflations) are directed colimits of split exact sequences in $R\textnormal{-Mod}$. The exact category $(R\textnormal{-Mod},\mathscr{E}_\mathrm{pure})$ admits two different exact full embeddings into abelian categories. One is the Yoneda embedding,
\begin{equation*}
  \label{eq:Yoneda}
  (R\textnormal{-Mod},\mathscr{E}_\mathrm{pure}) \longrightarrow [(R\textnormal{-mod})^\mathrm{op},\Ab]_0
  \qquad \textnormal{given by} \qquad 
  M \longmapsto \Hom_R(-,M)|_{R\textnormal{-mod}}\;;
\end{equation*}
the other is the so-called tensor embedding,
\begin{equation}
  \label{eq:Tensor}
  (R\textnormal{-Mod},\mathscr{E}_\mathrm{pure}) \longrightarrow [\textnormal{mod-}R,\Ab]_0
  \qquad \textnormal{given by} \qquad 
  M \longmapsto (- \otimes_R M)|_{\textnormal{mod-}R}\;.
\end{equation}
Here ``$\textnormal{mod}$'' means finitely presentable modules and $[\mathcal{X},\Ab]_0$ denotes the category of additive functors from $\mathcal{X}$ to the category $\Ab$ of abelian groups. For a detailed discussion and proofs of these embeddings we refer to \cite[Thms. B.11 and B.16]{JL}.

We point out some important and interesting generalizations of the Yoneda embedding, mentioned above, that can be found in the literature and have motivated this work.
\begin{prt}

\smallskip

\item[($*$)] Any locally finitely presentable (= locally $\aleph_0$-presentable) abelian\footnote{\,The category need not be abelian; it suffiecs to assume that it is additive and idempotent complete.} category $\mathcal{C}$ can~be equipped with the \emph{categorically pure exact structure}, $\mathscr{E}_{\aleph_0}$, consisting of exact sequences $0 \to X \to Y \to Z \to 0$ in $\mathcal{C}$ for which the sequence
\begin{equation*}
  0 \longrightarrow \Hom_\mathcal{C}(C,X) \longrightarrow \Hom_\mathcal{C}(C,Y) \longrightarrow \Hom_\mathcal{C}(C,Z) \longrightarrow 0
\end{equation*}
is exact in $\Ab$ for every finitely presentable (= $\aleph_0$-presentable) object $C \in \mathcal{C}$. In this case, the Yoneda functor,
\begin{equation*}
  \qquad
  (\mathcal{C},\mathscr{E}_{\aleph_0}) \longrightarrow [\fp{\mathcal{C}}^\mathrm{op},\Ab]_0
  \qquad \textnormal{given by} \qquad 
  X \longmapsto \Hom_\mathcal{C}(-,X)|_{\fp{\mathcal{C}}}\;,
\end{equation*}
is an exact full embedding whose essential image is the subcategory of flat functors (= directed colimits of representable functors). Furthermore, the Yoneda embedding identifies the pure projective objects in $\mathcal{C}$ (= the objects in $\mathcal{C}$ that are projective relative to the exact structure $\mathscr{E}_{\aleph_0}$) with the projective objects in $[\fp{\mathcal{C}}^\mathrm{op},\Ab]_0$. These results can be found in Cravley-Boevey \cite[(1.4) and \S3]{CB}, but see also Lenzing \cite{Lenzing}. 

Note that for $\mathcal{C} = R\textnormal{-Mod}$ the categorically pure exact structure $\mathscr{E}_{\aleph_0}$ coincides with the pure exact structure $\mathscr{E}_\mathrm{pure}$ mentioned previously; see \cite[Thm.~6.4]{JL}. One advantage of the identifications provided by the Yoneda embedding is that $\mathcal{C}$ is equivalent to the category of flat unitary modules over a (non-unital) ring with enough idempotents. For further applications of this embedding see for example \cite{BCE19} and \cite{Sto14}.

More generally, if $\mathcal{C}$ is a locally $\lambda$-presentable abelian category, where $\lambda$ is a~regular cardinal, then it can be equipped with a \emph{categorically pure exact structure}, $\mathscr{E}_{\lambda}$, which is defined similarly to $\mathscr{E}_{\aleph_0}$ and treated in \cite{AR04} by Ad\'{a}mek and Rosick\'{y} (see also the discussion preceding \stpref{setup-abelian-cosmos-1}). Also in this case, the Yoneda functor 
\begin{equation*}
  \qquad 
  (\mathcal{C},\mathscr{E}_{\lambda}) \longrightarrow [\lPres{\mathcal{C}}^\mathrm{op},\Ab]_0
  \qquad \textnormal{given by} \qquad 
  X \longmapsto \Hom_\mathcal{C}(-,X)|_{\lPres{\mathcal{C}}}\;,
\end{equation*}
is an exact full embedding, where $\lPres{\mathcal{C}}$ is the category of $\lambda$-presentable objects. 

\smallskip

\item[($**$)] The Yoneda embeddding has also been studied in the context of enriched categories. Let $\V$ be a locally $\lambda$-presentable base and let $\mathcal{C}$ be a locally $\lambda$-presentable $\V$-category in the sense of Borceux, Quinteiro, and Rosick\'{y} \cite[Dfns.~1.1 and 6.1]{BQR98}. Denote by $\lPres{\mathcal{C}}$ the full $\V$-subcategory of $\lambda$-pre\-sent\-able objects in $\mathcal{C}$, in the enriched sense \cite[Dfn.~3.1]{BQR98}, and let $[\lPres{\mathcal{C}}^\mathrm{op},\V]$ be the $\V$-category of $\V$-functors from $\lPres{\mathcal{C}}^\mathrm{op}$ to $\V$. In \cite[(proof of) Thm.~6.3]{BQR98} it is shown that the Yoneda $\V$-functor
\begin{equation*}
  \upUpsilon \colon \mathcal{C} \longrightarrow [\lPres{\mathcal{C}}^\mathrm{op},\V]
  \qquad \textnormal{given by} \qquad 
  X \longmapsto \mathcal{C}(-,X)|_{\lPres{\mathcal{C}}}
\end{equation*}
is fully faithful with essential image:
\begin{equation*}
  \operatorname{Ess.Im}\upUpsilon \,=\, \lFlat{\lPres{\mathcal{C}}^\mathrm{op}}{\V} \,=\,
  \lCon{\lPres{\mathcal{C}}^\mathrm{op}}{\V}\;.
\end{equation*}
Here $\lFlat{\lPres{\mathcal{C}}^\mathrm{op}}{\V}$ is the $\V$-subcategory of $[\lPres{\mathcal{C}}^\mathrm{op},\V]$ consisting of $\lambda$-flat $\V$-functors, in the enriched sense, and $\lCon{\lPres{\mathcal{C}}^\mathrm{op}}{\V}$ is the $\V$-subcategory of $\lambda$-continuous $\V$-functors, that is, $\V$-functors that preserve $\lambda$-small $\V$-limits.

\smallskip

\end{prt}

In contrast to the Yoneda embedding, the tensor embedding \eqref{Tensor} and its possible generalizations have only received little attention in the literature. One reason for this is probably that any potential generalization\,/\,extension of \eqref{Tensor} within ordinary category theory seems impossible, as the definition itself requires the existence of a suitable tensor product. However, it is possible to make sense of the tensor embedding for a closed symmetric~mo\-noid\-al category, and this is exactly what we do in this paper. More precisely, we consider to begin with (in \secref{tensor-embedding}) an abelian cosmos $(\V,\otimes,I,[-,-])$ and the $\V$-functor
\begin{equation*}
\upTheta \colon \V \longrightarrow [\A,\V]
\qquad \text{given by} \qquad 
X \longmapsto (X\otimes-)|_{\A}\;,
\end{equation*}
where $\A$ is any full $\V$-subcategory of $\V$ containing the unit object $I$. We call $\upTheta$ the \emph{tensor embedding} and we show in \thmref{main} that it is, indeed, fully faithful, and thus it induces an equivalence of $\V$-categories $\V \simeq \operatorname{Ess.Im}\upTheta$. We also prove that $\upTheta$ preserves $\V$-colimits.

Certainly, $\upTheta$ induces an (ordinary) additive functor
\begin{equation}
\label{eq:ord-ten}
\upTheta_0 \colon \V_0 \longrightarrow [\A,\V]_0
\end{equation}
between the underlying abelian categories (the fact that $[\A,\V]_0$ is abelian is contained in \cite[Thm.~4.2]{AG16} by Al~Hwaeer and Garkusha). As we now explain, this functor is intimately connected with the notion of geometrically purity.

As $\V$ is closed symmetric monoidal, it can be equipped with the so-called \emph{geometrically pure exact structure}, $\mathscr{E}_{\otimes}$, in which the admissible monomorphisms are geometrically pure monomorphisms introduced by Fox \cite{Fox} (see \dfnref{geo-exact-structure}). The exact category $(\V_0,\mathscr{E}_{\otimes})$ has recently been studied 
in works of Enochs, Estrada, Gillespie, and Odaba\c{s}{\i} \cite{EEO16,EGO17}, and we continue to investigate it in this paper. Note that if  $\V$ happens to be locally $\lambda$-presentable (which will often be the case), then it also makes sense to consider the \emph{categorically pure exact structure}, $\mathscr{E}_{\lambda}$, from ($*$). As mentioned in \cite[Rem.~2.8]{EGO17}, one always has $\mathscr{E}_{\lambda} \subseteq \mathscr{E}_{\otimes}$, but in general these two exact structures are different! However, for $\V=\Mod{R}$ they agree by \cite[Thm.~6.4]{JL}. Although being different from the categorically pure exact structure, the geometrically pure exact structure, $\mathscr{E}_{\otimes}$, captures many interesting notions of purity, e.g.:
\begin{itemlist}

\item The category $\Ch{R}$ of chain complexes of $R$-modules ($R$ is any commutative ring) is closed symmetric monoidal when equipped with the total tensor product and total Hom. In this situation, a short exact sequence $0 \to C' \to C \to C'' \to 0$ is in $\mathscr{E}_{\otimes}$ if and only if it is degreewise pure exact, meaning that $0 \to C'_n \to C_n \to C''_n \to 0$ is a pure exact sequence of $R$-modules for every $n \in \mathbb{Z}$. See \exaref{Ch-1}\prtlbl{a}.

\item In the closed symmetric monoidal category $\Qcoh{X}$ of quasi-coherent sheaves on a quasi-seperated scheme $X$, a short exact sequence $0 \to F' \to F \to F'' \to 0$ is in $\mathscr{E}_{\otimes}$ if and only if   it is stalkwise pure exact, meaning that $0 \to F'_x \to F_x \to F''_x \to 0$ is a pure exact sequence of $\mathscr{O}_{X,x}$-modules for every $x \in X$. See \exaref{QcohX-1}.
\end{itemlist}

In \secref{purity} we study purity. A main result about the geometrically pure exact category, which we prove in 
\prpref{dual2} and \thmref{pureinjectiveenvelope}, is the following.

\begin{res*}[Theorem~A]
  The exact category $(\V_0,\mathscr{E}_{\otimes})$ has enough relative injectives. In the language of relative homological algebra, this means that every object in $\V_0$ has a geometrically pure injective \emph{\emph{preenvelope}}. If\, $\V_0$ is Grothendieck, then every object in $\V_0$ even has a geometri\-cally pure injective \emph{\emph{envelope}}.
\end{res*}

In \dfnref{star-exact-structure}\,/\,\prpref{prp-star} we introduce a certain exact structure on the abelian category $[\A,\V]_0$. We call it the \emph{$\star$-pure exact structure}, $\mathscr{E}_{\star}$, and it is usually strictly coarser than the exact structure induced by the abelian structure on $[\A,\V]_0$. As already hinted, there is a connection between the tensor embedding functor $\upTheta_0$ from \textnormal{\eqref{ord-ten}} and the geometrically pure exact structure on $\mathcal{V}_0$. The main result in \secref{tensor-embedding} is \thmref{main}, which contains:
 
\begin{res*}[Theorem~B]
  The tensor embedding yields a fully faithful exact functor,
  \begin{equation*}
    \upTheta_0 \colon (\V_0,\mathscr{E}_{\otimes}) \longrightarrow ([\A,\V]_0,\mathscr{E}_{\star})\;,
  \end{equation*}
  which induces an equivalence of exact categories $(\V_0,\mathscr{E}_{\otimes}) \simeq (\operatorname{Ess.Im}\upTheta,\mathscr{E}_\star|_{\operatorname{Ess.Im}\upTheta})$.
\end{res*}

So far (i.e.~in \secref[Sections~]{purity} and \secref[]{tensor-embedding}) $\V$ has been an abelian cosmos and $\A$ has been \textsl{any} full $\V$-subcategory of $\V$ containing the unit object $I$. In \secref{Grothendieck-cosmos} we assume that $\V$ is~a~Gro\-then\-dieck cosmos. We show in \prpref{Grothendieck-lpb} that there exists some regular cardinal $\lambda$ for which 
$\V$ is a locally $\lambda$-presentable base, and we focus now only on the case where
\begin{equation*}
  \A \,=\, \lPres{\V}
\end{equation*}
is the the $\V$-subcategory of $\lambda$-presentable objects in $\V$ (in the ordinary categorical sense, or in the enriched sense; it makes no difference by \cite[Cor.~3.3]{BQR98}). In this situation, we explicitly describe the essential image of $\upTheta$. The description is, in some sense, dual to~the one for the Yoneda embedding in ($**$) above. We also show that in this case the $\star$-pure exact structure, $\mathscr{E}_{\star}$, and the abelian exact structure from $[\lPres{\V},\V]_0$ agree on $\operatorname{Ess.Im}\upTheta$, and that simplifies the last statement in Theorem~B. The precise statements are given below; they are contained in \thmref{main2}, which is the main result of \secref{Grothendieck-cosmos}.

\begin{res*}[Theorem~C]
  The essential image of the fully faithful tensor embedding
\begin{equation*}
\upTheta \colon \V \longrightarrow [\lPres{\V},\V]
\qquad \text{given by} \qquad 
X \longmapsto (X\otimes-)|_{\lPres{\V}}\;
\end{equation*}  
is precisely $\operatorname{Ess.Im}\upTheta = \lCocon{\lPres{\V}}{\V}$, that is, the subcategory of $\lambda$-cocontinuous $\V$-functors from $\lPres{\V}$ to $\V$. Further, $\upTheta_0$ induces an equivalence of exact categories,
\begin{equation*}
  (\V_0,\mathscr{E}_{\otimes}) \,\simeq\, \lCocon{\lPres{\V}}{\V}\;;
\end{equation*} 
the exact structure on the right-hand side is induced by the abelian structure on $[\lPres{\V},\V]_0$.
\end{res*}

In the final \secref{dualizable} we specialize the setup even further. Here we require $\V$ to be a Grothendieck cosmos (as in \secref{Grothendieck-cosmos}) which is generated by  a set of \emph{dualizable} objects and where the unit object $I$ is finitely presentable. The category $\Ch{R}$ of chain complexes always satisfies these requirements, and so does the category $\Qcoh{X}$ of quasi-coherent sheaves for most schemes $X$ (see \exaref[Examples~]{Ch-new} and \exaref[]{QcohX-new}). We prove in \prpref{aleph0-base} that such a category $\V$ is a locally finitely presentable base, which means that we can apply our previous results with $\lambda = \aleph_0$. In this case, 
\begin{equation*}
  \lPres[\aleph_0]{\V} \,=:\, \fp{\V}
\end{equation*}
is the class of finitely presentable objects in $\V$. The main result in the last section is~\thmref{main3}, of which the following is a special case:

\begin{res*}[Theorem~D]
  For $\V$ as described above, the tensor embedding from Theorem~C with~\mbox{$\lambda = \aleph_0$} restricts to an equivalence between the geometrically pure injective objects in $\V_0$ and the (categorically) injective objects in $[\fp{\V},\V]_0$. In symbols:
\begin{equation*}
\PureInj{\V_0} \,\simeq\, \Inj{[\fp{\V},\V]_0}\;.
\end{equation*}  
\end{res*}

\medskip

  This work has been developed in the setting of a suitably nice abelian cosmos $\V$. Unfortunately, this setting excludes applications to the ``non-commutative realm'', in particular, it does not cover the original tensor embedding \eqref{Tensor}. However, it is possible to develop much of the theory, not just for the category $\V$, but for the category $R\textnormal{-Mod}$ of \emph{$R$-left-objects} (or \emph{left $R$-modules}) in the sense of Pareigis \cite{Par77}, where $R$ is any \emph{monoid} (or \emph{ring object})~in~$\V$. Note that $\V$ is a special case of $R\textnormal{-Mod}$ as the unit object $I$ is a commutative monoid in $\V$ with $I\textnormal{-Mod} = \V$. To develop the theory found in this paper for $R\textnormal{-Mod}$ instead of just $\V$, one basically uses the same proofs, but things become more technical. A reader who wants to carry out this program should be able to do so with the information given in \mbox{\rmkref{monoid}}.

\section{Preliminaries}

We recall some definitions and terminology from enriched category theory that are important in this paper. We also a give some examples, which we shall repeatedly return~to.

\begin{bfhpg}[Locally presentable categories (\cite{AR})] 
\label{present}
Let  $\lambda$ be a regular cardinal. An object $A$ in a category $\V$ is said to be \emph{$\lambda$-presentable} if the functor $\V_0(A,-) \colon \V \to \Set$
preserves \mbox{$\lambda$-directed} colimits. One says that $\V$ is \emph{locally $\lambda$-presentable} if it is cocomplete and there is a
\textsl{set} $\mathcal S$ of $\lambda$-presentable objects such that every object in $\V$ is a $\lambda$-directed colimit of objects in $\mathcal{S}$.

It is customary to say \emph{finitely presentable} instead of ``$\aleph_0$-pre\-sent\-able''; thus $\aleph_0$-pre\-sent\-able objects are called \emph{finitely presentable objects} and locally $\aleph_0$-presentable categories are called \emph{locally finitely presentable categories}. Moreover, ``$\aleph_0$-directed colimits'' are~simply called \emph{directed colimits}.
\end{bfhpg}

\begin{bfhpg}[Monoidal categories (\cite{Kelly})] 
\label{monoidal-category} 
A \textit{monoidal category} consists of a category $\V$, a bifunctor $\otimes \colon \V \times \V \rightarrow \V$ (tensor product), a unit object $I \in \V$, and natural isomorphisms $a$ (associator), $l$ (left unitor), and $r$ (right unitor) subject to the coherence axioms found in \cite[\S1.1 eq.~(1.1) and (1.2)]{Kelly}. A monoidal category $\V$ is said to be
\emph{symmetric} if there is a natural isomorphism $c$ (symmetry) subject to further coherence axioms that express the compatibility of $c$ with $a$, $l$, and $r$; see \cite[\S1.4 eq.~(1.14)--(1.16)]{Kelly}. In particular, the symmetry $c$ identifies $l$ and $r$ so there is no need to distinguish between them. Due to Mac Lane's \textit{coherence theorem}, see \cite{MacLane63} or \cite[Sect.~VII.2]{MacLane}, it is customary to suppress $a$, $l$, $r$, and $c$, and we simply write $(\V,\otimes,I)$ when referring to a (symmetric) monoidal category.
A symmetric monoidal category is said to be \emph{closed} if for every $X \in \V$, the functor $-\otimes X \colon \V \rightarrow \V$ has a right adjoint $[X,-] \colon \V \rightarrow \V$; see
\cite[\S1.5]{Kelly}. It turns out that $[-,-]$ is a bifunctor $\V^{\mathrm{op}} \times \V \rightarrow \V$, and we write a closed symmetric monoidal category as a quadruple $(\V,\otimes,I,[-,-])$.
\end{bfhpg}

\begin{exa}
  \label{exa:Ch}
  The category $\Ch{R}$ of chain complexes of modules over a commutative ring $R$ is a Grothendieck category with two different closed symmetric monoidal structures:
  \begin{prt}
  \item $(\Ch{R},\tTen,\stalk{R},\tHom)$ where $\tTen$ is the \emph{total tensor product}, $\tHom$ the \emph{total Hom}, and $\stalk{R} = 0 \to R \to 0$ is the \emph{stalk} complex with $R$ in degree $0$. See \cite[App.~A.2]{Chr}.
  \item $(\Ch{R},\mTen,\disc{R},\mHom)$ where $\mTen$ is the \emph{modified} total tensor product, $\mHom$ the \emph{modified} total Hom, and $\disc{R} = 0 \to R \to R \to 0$ is the \emph{disc} complex concentrated in homological degrees $0$ and $-1$. See \cite[\S2]{EG97} or \cite[\S4.2]{MR1693036}.
  \end{prt}
\end{exa}

\begin{exa}
  \label{exa:QcohX}
  Some important examples of closed symmetric monoidal categories, which are also Grothendieck, come from algebraic geometry. Let $X$ be any scheme.
  \begin{prt}
  \item $(\Mod{X},\otimes_{X},\mathscr{O}_X,\shHome_X)$ is a closed symmetric mo\-noi\-dal category, where $\Mod{X}$ is the abelian category of all sheaves (of $\mathscr{O}_X$-mod\-ules) on $X$; see \cite[~II\S5]{Hartshorne}. It is well-known that this is a Grothendieck category; see \cite[Prop.~3.1.1]{AGr57}.

  \item The category $\Qcoh{X}$ of quasi-coherent sheaves on $X$ is an abelian subcategory of $\Mod{X}$; see \cite[Cor.~(2.2.2)(i,ii)]{MR3075000} or \cite[II Prop.~5.7]{Hartshorne}. As $I = \mathscr{O}_X$ is quasi-coherent and quasi-coherent sheaves are closed under tensor products by \cite[Cor.~(2.2.2)(v)]{MR3075000}, it follows that $(\Qcoh{X},\otimes_{X},\mathscr{O}_X)$ is a monoidal subcategory of $(\Mod{X},\otimes_{X},\mathscr{O}_X)$. In ge\-ne\-ral, $\shHome_X$ is not an internal hom in $\Qcoh{X}$. However, the inclusion functor \mbox{$\Qcoh{X} \to \Mod{X}$} admits a right adjoint $Q_X \colon \Mod{X} \to \Qcoh{X}$, called~the \emph{coherator}, and the counit $Q_X(\mathscr{F}) \to \mathscr{F}$ is an isomorphism for every quasi-coherent sheaf $\mathscr{F}$; see \cite[Tag~08D6]{stacks-project}. It is well-known, and completely formal, that the functor \smash{$\shHomqce_X := Q_X\shHome_X$} yields a closed structure on $(\Qcoh{X},\otimes_{X},\mathscr{O}_X)$. The category $\Qcoh{X}$ is Grothendieck by \cite[Lem.~1.3]{AJPV08}.
  \end{prt}
\end{exa}

A category can be locally presentable (as in \ref{present}) and closed symmetric monoidal (as in \ref{monoidal-category}) at the same time, but in generel one can not expect any compatibility between the two structures. This is the reason for the next definition, which comes from \cite{BQR98}.

\begin{bfhpg}[Locally presentable bases {(\cite[Dfn.~1.1]{BQR98})}] \label{lpb}
Let $\lambda$ be a regular cardinal.  A closed symmetric monoidal category $(\V,\otimes,I,[-,-])$ is said to be a \emph{locally $\lambda$-presentable base} if it satisfies the following condition:
\begin{rqm}
\item the category $\V$ is locally $\lambda$-presentable,
\item the unit object $I$ is $\lambda$-presentable, and
\item the class of $\lambda$-presentable objects is closed under the tensor product $\otimes$.
\end{rqm}
A locally $\aleph_0$-presentable base is simply called a \emph{locally  finitely presentable base}.
\end{bfhpg}

\begin{bfhpg}[Enriched category theory] 
  \label{ect}
  We assume familiarity with basic notions and results from enriched cateory theory as presented in \cite[Chaps.~1--2 (and parts of 3)]{Kelly}. In particular, for a closed symmetric monoidal category $(\V,\otimes,I,[-,-])$, the definitions and properties of \emph{$\V$-categories} and their \emph{underlying ordinary categories}, \emph{$\V$-functors}, \emph{$\V$-natural transformations}, and \emph{weighted limits and colimits} will be important. When we use specific results from enriched category theory, we will give appropriate references to \cite{Kelly}, but a~few general points are mentioned below.
  
  To avoid confusion, we often write $\V_0$ when we think of $\V$ as an ordinary category, and we use the symbol $\V$ when it is viewed as a $\V$-category.

  If the category $\V_0$ is complete, $\mathcal{K}$ is a small $\V$-category, and $\mathcal{C}$ is any $\V$-category, there is a $\V$-category $[\mathcal{K},\mathcal{C}]$ whose objects are $\V$-functors $\mathcal{K} \to \mathcal{C}$. The underlying ordinary category $[\mathcal{K},\mathcal{C}]_0$ has $\V$-natural transformations as morphisms. See \cite[\S2.1--2.2]{Kelly}.
  
  In the proof of \prpref{essim} we use the \emph{unit $\V$-category $\mathcal{I}$}. It has one object, $*$, and $\mathcal{I}(*,*)=I$. The composition law is given by the isomorphism $I\otimes I \rightarrow I$. See \cite[\S1.3]{Kelly}.
\end{bfhpg}

In \dfnref[]{lambda-small-functor}--\ref{tensored-cotensored} below, $(\V,\otimes,I,[-,-])$ denotes a \emph{cosmos}, that is, a closed symmetric monoidal category for which $\V_0$ is bicomplete. The examples in \exaref[]{Ch} and \exaref[]{QcohX} are all cosmos.

The next notion of smallness for a $\V$-functor with values in $\V$ will be important to us.

\begin{dfn}(\cite[Dfn.~2.1]{BQR98}). \label{dfn:lambda-small-functor} Let $\lambda$ be a regular cardinal. A $\V$-functor $T \colon \mathcal{K} \rightarrow  \V$ is said to be \textit{$\lambda$-small} if the following conditions are satisfied:
\begin{rqm}
\item the class  $\operatorname{Ob}\mathcal{K}$ is a set of cardinality strictly less than $\lambda$,
\item for all objects $X,Y \in \mathcal{K}$, the hom-object $\mathcal{K}(X,Y)$ is $\lambda$-presentable in $\V_0$, and 
\item for every object $X \in \mathcal{K}$, the object $T(X)$ is $\lambda$-presentable in $\V_0$.
\end{rqm}
\end{dfn}	

We shall also need the ``enriched versions'' of limits and colimits:

\begin{bfhpg}[Weighted limits and colimits \mbox{(\cite[Chap.~3]{Kelly})}] \label{weighted} Let $F \colon \mathcal{K} \rightarrow \V$ and $G \colon \mathcal{K} \rightarrow \A$ be $\V$-functors. The \textit{$\V$-limit of $G$ weighted by $F$}, if it exists, is an object $\{F,G\} \in \A$ for which there is a $\V$-natural isomorphism in $A \in \A$:
\begin{equation*}
\label{eq:eq-wlim}
\A(A,\{F,G\}) \,\cong\, [\mathcal{K},\V](F,\A(A,G(-)))\;.
\end{equation*}

Given $\V$-functors $G \colon \mathcal{K}^{\mathrm{op}} \rightarrow \V$ and $F \colon \mathcal{K} \rightarrow \A$, the \textit{$\V$-colimit of $F$ weighted by $G$}, if it exists, is an object $G \star F \in \A$ for which there is a $\V$-natural isomorphism in $A \in \A$: 
\begin{equation}
\label{eq:eq-wcolim}
\A(G \star F, A) \,\cong\, [\mathcal{K}^{\mathrm{op}},\V](G, \A(F(-),A))\;.
\end{equation}
\end{bfhpg}

\begin{bfhpg}[Tensors and cotensors \mbox{(\cite[Prop.~6.5.7]{B2})}] \label{tensored-cotensored} Let $\mathcal{K}$ be a small $\V$-category. The $\V$-category $[\mathcal{K},\V]$ is both \emph{tensored} and \emph{cotensored}. By \cite[Dfn. 6.5.1]{B2} this means that for every $V \in \V$ and $F \in [\mathcal{K},\V]$ there exist objects $V \otimes F$ and $[V,F]$ in $[\mathcal{K},\V]$ and $\V$-natural isomorphisms,
\begin{equation}
  \label{eq:eq-tensored-cotensored}
  [\mathcal{K},\V](V \otimes F,?) \,\cong\, [V, [\mathcal{K},\V](F,?)]
  \quad \textnormal{ and } \quad
  [\mathcal{K},\V](?,[V,F]) \,\cong\, [V, [\mathcal{K},\V](?,F)]\;.
\end{equation}
The proof of \cite[Prop.~6.5.7]{B2} reveals that the $\V$-functors $V \otimes F$ and $[V,F]$ are just the compositions $V \otimes F = (V \otimes -) \circ F$ and $[V,F] = [V,-] \circ F$.
\end{bfhpg}

\section{Exact categories and purity}
\label{sec:purity}

We demonstrate (\prpref{construction-of-exact-structure}) a general procedure to construct exact structures on an abelian category, and apply it to establish the so-called geometrically pure exact structure on $\V_0$ (\dfnref{geo-exact-structure}) and the $\star$-pure exact structure on $[\mathcal{K},\V]_0$ (\dfnref{star-exact-structure}).

\begin{bfhpg}[Exact categories (\cite{Quillen})] \label{exact-cat} Let $\mathcal{X}$ be an additive category and $\mathscr{E}$ be a class of  kernel-cokernel pairs $(i,p)$ in $\mathcal{X}$,
\begin{equation*}
\xymatrix{X \ar@{>->}[r]^-{i} & Y \ar@{->>}[r]^-{p} & Z},
\end{equation*}
that is,  $i$ is the kernel of $p$, and $p$ is the cokernel of $i$. The morphism $i$ is called an \emph{admissible monic} and $p$ an \emph{admissible epic} in $\mathscr{E}$.  The class $\mathscr{E}$ is said to form an \textit{exact structure} on $\mathcal{X}$ if it is closed under isomorphisms and satisfies the following axioms:
\begin{rqm}
	\item[(E0)] For every object $X$ in  $\mathcal{X}$, the identity morphism $\mathrm{id}_X$ is both an admissible monic and an admissible epic in $\mathscr{E}$.
	\item[(E1)] The classes of admissible monics and admissible epics in $\mathscr{E}$ are closed under compositions.
	
	\item[(E2)] The pushout (resp., pullback) of an admissible monic (resp., admissible epic) along an arbitrary morphism exists and yields an admissible monic (resp., admissible epic). 
\end{rqm}
In this situation, the pair $(\mathcal{X}, \mathscr{E})$ is called an \textit{exact category}. An object in $J \in \mathcal{X}$ is said to be \emph{injective relative to $\mathscr{E}$} if the functor $\Hom_\mathcal{X}(-,J)$ maps sequences in $\mathscr{E}$ to short exact sequences in $\Ab$.
For a detailed treatment on the subject, see \cite{Buhler}.
\end{bfhpg}

We begin with a general result, potentially of independent interest, which shows how to construct an exact structure $\mathscr{E}_{\mathfrak{T}}$ on an abelian category $\mathcal{C}$ from a collection $\mathfrak{T}$ of functors. Inspired by terminology from topology, we call $\mathscr{E}_{\mathfrak{T}}$ the \emph{initial exact structure on $\mathcal{C}$ w.r.t. $\mathfrak{T}$}. 

\begin{prp}
\label{prp:construction-of-exact-structure} 
Let $\mathcal{C}$ be an abelian category and $\mathfrak{T}$ a collection of additive functors \mbox{$T \colon \mathcal{C} \to \mathcal{D}_T$} where each category $\mathcal{D}_T$ is abelian and each functor $T$ is left exact or right exact. Denote by $\mathscr{E}_{\mathfrak{T}}$  the class of  all short exact sequences $0 \to X \to Y \to Z \to 0$ in $\mathcal{C}$~such~that 
\begin{equation*}
0 \longrightarrow TX \longrightarrow TY \longrightarrow TZ \longrightarrow 0
\end{equation*} 
is exact in $\mathcal{D}_T$ for every $T$ in $\mathfrak{T}$. Then $\mathscr{E}_{\mathfrak{T}}$ is an exact structure on $\mathcal{C}$, in fact, it is the finest (that is, the largest w.r.t.~inclusion) exact structure $\mathscr{E}$ on $\mathcal{C}$ which satisfies the condition that $T \colon (\mathcal{C},\mathscr{E}) \to \mathcal{D}_T$ is an exact functor for every $T$ in $\mathfrak{T}$. 
\end{prp}

\begin{proof}
Once we have proved that $\mathscr{E}_{\mathfrak{T}}$ is, in fact, an exact structure on $\mathcal{C}$, then certainly $T \colon (\mathcal{C},\mathscr{E}_{\mathfrak{T}}) \to \mathcal{D}_T$ is an exact functor for every $T$ in $\mathfrak{T}$. Moreover, if $\mathscr{E}$ is any exact structure on $\mathcal{C}$ for which every $T$ in $\mathfrak{T}$ is an exact functor $T \colon (\mathcal{C},\mathscr{E}) \to \mathcal{D}_T$, then evidently $\mathscr{E} \subseteq \mathscr{E}_{\mathfrak{T}}$.

We now show that $\mathscr{E}_{\mathfrak{T}}$ satisfies the axioms in \ref{exact-cat}. The condition (E0) is immediate from the definition of $\mathscr{E}_{\mathfrak{T}}$. To show (E1), let $f \colon X \to Y$ and $g \colon Y \to Z$ be composable morphisms in $\mathcal{C}$. We will prove that if $f$ and $g$ are admissible monics in $\mathscr{E}_{\mathfrak{T}}$, then so is $gf$. The case where $f$ and $g$ are admissible epics in $\mathscr{E}_{\mathfrak{T}}$ is proved similarly. If $f$ and $g$ are admissible monics in $\mathscr{E}_{\mathfrak{T}}$ then, by definition, $f$ and $g$ are monics in $\mathcal{C}$ and the short exact sequences
\begin{equation*}
  \xymatrix@C=1.5pc{
    0 \ar[r] & X \ar[r]^-{f} & Y \ar[r] & \Coker{f} \ar[r] & 0
  }
  \qquad \textnormal{and} \qquad
  \xymatrix@C=1.5pc{
    0 \ar[r] & Y \ar[r]^-{g} & Z \ar[r] & \Coker{g} \ar[r] & 0
  }
\end{equation*}
stay exact under every functor $T$ in $\mathfrak{T}$. The composition $gf$ is certainly a monic in $\mathcal{C}$, so it remains to prove that the short exact sequence 
\begin{equation}
  \label{eq:eq-gf}
  \xymatrix@C=1.5pc{
    0 \ar[r] & X \ar[r]^-{gf} & Z \ar[r] & \Coker{(gf)} \ar[r] & 0
  }
\end{equation}
stays exact under every functor $T$ in $\mathfrak{T}$. Let $T$ in $\mathfrak{T}$ be given and recall that $T$ is assumed to be left exact or right exact. As the composition of two monics, $T(gf) = T(g)T(f)$ is a monic. Thus, if $T$ is right exact, the sequence \eqref{eq-gf} certainly stays exact under $T$. Assume that $T$ is left exact. In the leftmost commutative diagram below, the lower row is exact by the Snake Lemma; the remaining rows and all columns are trivially exact. The rightmost commutative diagram is obtained by applying the functor $T$ to the leftmost one. In the right diagram, the $1^\mathrm{st}$ column and $2^\mathrm{nd}$ row are exact by assumption, and the $1^\mathrm{st}$ row and $3^\mathrm{rd}$ column are trivially exact. The epimorphism $T(Z) \twoheadrightarrow T(\Coker{g})$ in the $2^\mathrm{nd}$ row factorizes as $T(Z)\to T(\Coker{(gf)}) \to T(\Coker{g})$, and hence the last morphism $T(\Coker{(gf)}) \to T(\Coker{g})$ in the $3^\mathrm{rd}$ row is epic too. Since $T$ is left exact, the entire $3^\mathrm{rd}$ row is exact. Consequently, in the rightmost diagram below, all three rows and the $1^\mathrm{st}$ and $3^\mathrm{rd}$ columns are exact.
\begin{equation*}
  \begin{gathered}
  \xymatrix@R=1.3pc@C=0.55pc{
    {} & 0 \ar[d] & 0 \ar[d] & 0 \ar[d] & {}
    \\
    0 \ar[r] &
    X \ar[r]^-{=} \ar[d]_-{f} &
    X \ar[r] \ar[d]^-{gf} &
    0 \ar[d] \ar[r] &
    0
    \\
    0 \ar[r] &
    Y \ar[r]^-{g} \ar[d] &
    Z \ar[r] \ar[d] &
    \Coker{g} \ar[d]^-{=} \ar[r] &
    0     
    \\
    0 \ar[r] &
    \Coker{f} \ar[r] \ar[d] &
    \Coker{(gf)} \ar[r] \ar[d] &
    \Coker{g} \ar[r] \ar[d] &
    0
    \\
    {} & 0 & 0 & 0 & {}    
  }
  \end{gathered}
  \qquad
  \begin{gathered}
  \xymatrix@R=1.3pc@C=0.55pc{
    {} & 0 \ar[d] & 0 \ar[d] & 0 \ar[d] & {}
    \\
    0 \ar[r] &
    T(X) \ar[r]^-{=} \ar[d]_-{T(f)} &
    T(X) \ar[r] \ar[d]^-{T(gf)} &
    0 \ar[d] \ar[r] &
    0
    \\
    0 \ar[r] &
    T(Y) \ar[r]^-{T(g)} \ar[d] &
    T(Z) \ar[r] \ar[d] &
    T(\Coker{g}) \ar[d]^-{=} \ar[r] &
    0     
    \\
    0 \ar[r] &
    T(\Coker{f}) \ar[r] \ar[d] &
    T(\Coker{(gf)}) \ar[r] \ar[d] &
    T(\Coker{g}) \ar[r] \ar[d] &
    0
    \\
    {} & 0 & 0 & 0 & {}    
  }  
  \end{gathered}
\end{equation*}
Thus, we can consider the rightmost diagram as an exact sequence $0 \to C_1 \to C_2 \to C_3 \to 0$ of complexes where $C_i$ is the $i^\mathrm{th}$ column in the diagram. As $C_1$ and $C_3$ are exact, so is $C_2$. Hence the sequence \eqref{eq-gf} stays exact under the functor $T$, as desired.

It remains to show (E2). We will show that 
the pushout of an admissible monic in $\mathscr{E}_{\mathfrak{T}}$ along an arbitrary morphism yields an admissible monic. A similar argument shows that the pullback of an admissible epic in $\mathscr{E}_{\mathfrak{T}}$ along an arbitrary morphism is an admissible epic. Thus, consider a pushout diagram in $\mathcal{C}$,
\begin{equation*}
\xymatrix{
  X
  \ar@{}[dr]|-{\mathrm{\scriptscriptstyle(pushout)}} \ar[d] \ar[r]^-{f} 
  & 
  Y \ar@{.>}[d]
  \\
  X' 
  \ar@{.>}[r]^-{f'} 
  & 
  Y'\;,\mspace{-8mu}
}
\end{equation*}
where $f$ is an admissible monic in $\mathscr{E}_{\mathfrak{T}}$ and $X \to X'$ is any morphism. As $f$ is, in particular, a mono\-morphism, so is $f'$ by \cite[Thm.~2.54*]{Freyd}, and hence there is a short exact sequence
\begin{equation*}
  \xymatrix@C=1.5pc{
    0 \ar[r] & X' \ar[r]^-{f'} & Y' \ar[r] & \Coker{f'} \ar[r] & 0
  }.
\end{equation*}
We must argue that this sequence stays exact under every $T$ in $\mathfrak{T}$. 

First assume that $T$ is right exact. In this case, $T(X') \to T(Y') \to T(\Coker{f'}) \to 0$ is exact, and it remains to see that $T(f')$ is monic. As $T$ preserves pushouts, $T(f')$ is a pushout of the monic $T(f)$, so another application of \cite[Thm.~2.54*]{Freyd} yields that $T(f')$~is~monic. 

Next assume that $T$ is left exact. In this case, $0 \to T(X') \to T(Y') \to T(\Coker{f'})$ is exact, and it remains to see that $T(Y') \to T(\Coker{f'})$ is epic. As $f'$ is a pushout of $f$, the canonical morphism $\Coker{f} \to \Coker{f'}$ is an isomorphism, cf. (the dual of) \cite[Thm.~2.52]{Freyd}, and hence so is $T(\Coker{f}) \to T(\Coker{f'})$. By assumption, $T(Y) \to T(\Coker{f})$ is epic, so the composite morphism \smash{$T(Y) \twoheadrightarrow T(\Coker{f}) \stackrel{\smash{\text{\raisebox{-2pt}{\smash{$\cong$}}}}}{\longrightarrow} T(\Coker{f'})$} is epic. But this composite is the same as the composite $T(Y) \to T(Y') \to T(\Coker{f'})$, which is therefore an epimorphism, and it follows that $T(Y') \to T(\Coker{f'})$ is an epimorphism.
\end{proof}

Any locally $\lambda$-presentable abelian category $\V$ (see \ref{present}) can be equipped with an exact structure (see \ref{exact-cat}) called the \emph{categorically pure exact structure} and denoted by $\mathscr{E}_{\lambda}$. In this exact structure, the admissible monomorphisms are precisely the \emph{$\lambda$-pure subobjects} and the admissible epimorphisms are precisely the \emph{$\lambda$-pure quotients} in the sense of \cite{AR04}. That these classes of morphisms do, in fact, yield an exact structure follows from Prop.~5, Obs.~11, and Prop.~15 in \emph{loc.~cit.}. Alternatively, it follows directly from  \prpref{construction-of-exact-structure} with 
$\mathcal{C}=\V$ and $\mathfrak{T}$ the collection of functors
$\V(A,-) \colon \V \to \Ab$ where $A$ ranges over the $\lambda$-presentable objects in $\V$. In the special case $\lambda = \aleph_0$, this kind of purity was studied in \cite[\S3]{CB}.

If $\V$ is a closed symmetric monoidal abelian category, there is also a notion of purity in $\V_0$ based on the tensor product (see \dfnref{geo-exact-structure}). In the literature, this kind of purity is often called \emph{geometrically purity} (as opposed to categorically purity, mentioned above). The study of geometrically purity was initiated in \cite{Fox} and was recently continued in \cite{EEO16} and \cite{EGO17}. 
Below we establish the \emph{geometrically pure exact structure}, $\mathscr{E}_\otimes$, on $\V_0$, and show that the exact category $(\V_0,\mathscr{E}_\otimes)$ has enough relative injectives (\prpref[Propositions~]{geometricalexactst} and \prpref[]{dual2}).

As mentioned in \cite[Rem.~2.8]{EGO17}, when both the categorically and the geometrically pure exact structures are available, the former is coarser than the latter, i.e.~one has
$\mathscr{E}_{\lambda} \subseteq \mathscr{E}_{\otimes}$. In general, this is a strict containment, however, in the locally finitely presentable categories $\V = \Mod{R}$, where $R$ is a commutative ring, one has $\mathscr{E}_{\aleph_0} = \mathscr{E}_{\otimes}$. See e.g.~\cite[Thm.~6.4]{JL}. As mentioned in \exaref{Ch-1}(b) below, this equality also holds for $\V=\Ch{R}$ with the modified total tensor product $\mTen$.

\medskip

Note that the examples found in \exaref[]{Ch} and \exaref[]{QcohX} all satisfy the following setup.

\begin{stp} 
\label{stp:setup-abelian-cosmos-1} 
In the rest of this section, $(\V,\otimes,I,[-,-])$ denotes a \emph{cosmos}, that is, a closed symmetric monoidal category which is bicomplete\footnote{\,Actually, we shall not use the bicompleteness of $\V_0$ until we get to \lemref{3T} and the subsequent results.}. We also assume that $\V_0$ is abelian\footnote{\,When we talk about an \textsl{abelian} closed symmetric monoidal category, we tacitly assume that the tensor product $-\otimes-$ and the internal hom $[-,-]$ are additive functors in each variable.} and that the category $\V_0$ has an injective cogenerator $E$. 
\end{stp}

Following Fox \cite{Fox} a morphism $f \colon X \rightarrow Y$ in $\V_0$ is said to be \textit{geometrically pure} if $f\otimes V \colon X\otimes V \rightarrow Y \otimes V$ is a monomorphism for every $V \in \V$. Note that a geometrically pure morphism is necessarily a monomorphism (take $V=I$).

\begin{dfn}
\label{dfn:geo-exact-structure} 
Let $\mathscr{E}_{\otimes}$ be the class of all short exact sequences in $\V_0$ which remain exact under the functor $-\otimes V$ for every $V \in \V$. We call $\mathscr{E}_{\otimes}$ the \emph{geometrically pure exact structure} on $\V_0$ (see \prpref{geometricalexactst} below).
Sequences in $\mathscr{E}_{\otimes}$ are called \textit{geometrically pure (short) exact sequences}. An object $J \in \V_0$ which is in\-jec\-tive relative to $\mathscr{E}_{\otimes}$ is called a \textit{geometrically pure injective} object. We set 
\begin{equation*}
  \PureInj{\V_0} \,=\, \{J \in \V_0 \,|\, \textnormal{$J$ is geometrically pure injective} \}\;.
\end{equation*}  
\end{dfn}

\begin{exa}
  \label{exa:Ch-1}
  Consider the abelian cosmos from \exaref{Ch}.
  \begin{prt}
  \item It is easy to see that a short exact sequence $\mathbb{S}$ in $\Ch{R}$ is  geometrically pure exact in $(\Ch{R},\tTen)$ if and only if $\mathbb{S}_n$ is a pure exact sequence of $R$-modules in each~degree~$n$. Therefore, geometrically pure injective objects in $(\Ch{R},\tTen)$ are precisely contractible chain complexes of pure injective $R$-modules; see \cite[Cor.~5.7]{Sto14}.
  
  \item The geometrically pure exact sequences in $(\Ch{R},\mTen)$ have been characterized in several ways in \cite[Thm.~2.5]{EG97} and \cite[Thm.~5.1.3]{MR1693036}. Namely, a short exact sequence $\mathbb{S}$ in $\Ch{R}$ is  geometrically pure exact in $(\Ch{R},\mTen)$ if and only if $\mathbb{S}$ is a categorically  pure exact sequence in $\Ch{R}$. Furthermore, if a chain complex $J$ of $R$-modules is a geometrically pure injective object in $(\Ch{R},\mTen)$, then $J_n$ and $\operatorname{Ker}\partial^J_n$ are pure injective $R$-modules for every integer $n$; see \cite[Prop.~5.1.4]{MR1693036}.
  \end{prt}
\end{exa}

\begin{exa}
  \label{exa:QcohX-1}
  Consider the abelian cosmos from \exaref{QcohX}\prtlbl{b}. For a quasi-seperated scheme $X$, a short exact sequence $\mathbb{S}$ is geometrically pure exact in $(\Qcoh{X},\otimes_{X})$ if and only if  $\mathbb{S}_x$ is a pure exact sequence of $\mathscr{O}_{X,x}$-modules for every $x \in X$. This is proved in \cite[Prop.~3.4 and Rem.~3.5]{EEO16}.
\end{exa}

\begin{prp}
\label{prp:geometricalexactst} 
The pair $(\V_0,\mathscr{E}_{\otimes})$ is an exact category.
\end{prp}

\begin{proof}
This is known and implicit in \cite[(proof of) Lem. 3.6]{EGO17}. It is also a special case of \prpref{construction-of-exact-structure} with $\mathcal{C} = \V_0$ and $\mathfrak{T}$ the class of functors $- \otimes V \colon \V_0 \to \V_0$~where~$V \in \V$.
\end{proof}

\begin{lem}
\label{lem:dual} 
For every $X  \in \V$, the object $[X,E]$ is geometrically pure injective.
\end{lem}

\begin{proof}
For any geometrically pure exact sequence $\mathbb{S}$, the sequence $\mathbb{S} \otimes X$ is exact, and hence so is $\V_0(\mathbb{S} \otimes X, E)$, as $E$ is injective. The isomorphism $\V_0(\mathbb{S},[X,E]) \cong \V_0(\mathbb{S} \otimes X, E)$ shows that $\V_0(\mathbb{S},[X,E])$ is exact, which means that $[X,E]$ is geometrically pure injective.
\end{proof}

\begin{lem}
\label{lem:gps1} 
A short exact sequence $\mathbb{S}$ in $\V_0$ is geometrically pure exact if and only if ~$\mspace{1.5mu}[\mathbb{S},E]$ is a split short exact sequence in $\V_0$.
\end{lem}

\begin{proof}
As $E$ is an injective cogenerator in $\V_0$, the sequence $\mathbb{S}$ is geometrically pure exact if and only if $\V_0(\mathbb{S} \otimes V, E)$ is a short exact sequence in $\Ab$ for every $V \in \V$. And $[\mathbb{S},E]$ is a split short exact sequence in $\V_0$ 
if and only if $\V_0(V,[\mathbb{S},E])$ is a short exact sequence in $\Ab$ for every $V \in \V$. The isomorphism $\V_0(\mathbb{S} \otimes V, E) \cong \V_0(V,[\mathbb{S},E])$ yields the conclusion.
\end{proof}

\begin{lem}
\label{lem:ev1} 
The functor $[-,E] \colon \V_0^{\mathrm{op}}\rightarrow \V_0$ is faithful.
\end{lem}

\begin{proof}
There is a natural isomorphism
$\V_0(-,E) \cong \V_0(I,[-,E])$. If $f \neq 0$ is a  morphism, then $\V_0(f,E) \neq 0$, as $E$ is a cogenerator in $\V_0$, so $\V_0(I,[f,E]) \neq 0$ and thus $[f,E] \neq 0$.  
\end{proof}

\begin{obs}
\label{obs:biduality} 
There is a pair of adjoint functors $(F,G)$ as follows:
\begin{equation*}
\xymatrix@C=5pc{
  \V_0
  \ar@<0.6ex>[r]^-{F\,=\,[-,E]}
  &
  \V_0^{\mathrm{op}}\;.\mspace{-8mu}
  \ar@<0.6ex>[l]^-{G\,=\,[-,E]}
}
\end{equation*}
Indeed, for all $X \in \V$ and $Y \in \V^{\mathrm{op}}$ (equivalently, $Y \in \V$) one has:
\begin{align*}
  \V_0^{\mathrm{op}}(F(X),Y) &\,=\,
  \V_0(Y,FX) \,=\, \V_0(Y,[X,E]) \,\cong\, \V_0(Y \otimes X,E)
  \\
  &\,\cong\, \V_0(X \otimes Y,E) \,\cong\, \V_0(X,[Y,E]) \,=\, \V_0(X,G(Y))\;.
\end{align*}  
Write $\varepsilon$ for the counit of the adjunction. For every object $Y$ in $\V$, note that $\varepsilon_Y$ is an element in $\V_0^{\mathrm{op}}(FG(Y),Y) = \V_0(Y,FG(Y))$, so $\varepsilon_Y$ is a morphism $Y \to FG(Y) = [[Y,E],E]$ in $\V_0$.
\end{obs}

\begin{prp}
\label{prp:dual2}
For every $Y \in \V$ the morphism $\varepsilon_Y \colon Y \to [[Y,E],E]$ from  \obsref{biduality} is a geometrically pure monomorphism. In particular, the exact category $(\V_0, \mathscr{E}_{\otimes})$ has enough relative injectives (= enough geometrically pure injectives).
\end{prp}

\begin{proof}
  First we show that $\varepsilon_Y$ is monic. Let $f$ be a morphism in $\V_0$ with  $\varepsilon_Y \circ f = 0$. It~fol\-lows that $[f,E] \circ [\varepsilon_Y,E] = 0$ in $\V_0$. By adjoint functor theory, see \mbox{\cite[\S IV.1 Thm.~1]{MacLane}}, the morphism $G(\varepsilon_Y) = [\varepsilon_Y,E]$ has a right-inverse (which is actually $\varepsilon_{[Y,E]}$, but this is not important), and hence $[f,E]=0$. Now \lemref{ev1} implies $f=0$, so $\varepsilon_Y$ is a monomorphism. To show that $\varepsilon_Y$ is a geometrically pure monomorphism, consider the short exact sequence 
\begin{equation*}
  \mathbb{S} \,=\,
  \xymatrix@C=1.5pc{
    0 \ar[r] & Y \ar[r]^-{\varepsilon_Y} & [[Y,E],E] \ar[r] & \Coker \varepsilon_Y \ar[r] & 0
  }.
\end{equation*}
By \lemref{gps1} we need to prove that $[\mathbb{S},E]$ splits, but as already argued above, $[\varepsilon_Y,E]$ is a split epimorphism, so we are done. To see that $(\V_0, \mathscr{E}_{\otimes})$ has enough relative injectives, it remains to note that $[[Y,E],E]$ is a geometrically pure injective object by \lemref{dual}.
\end{proof}

Recall from \cite[Dfn.~6.1.1]{rha} the notions of \emph{preenvelopes} and \emph{envelopes}.

\begin{thm}
\label{thm:pureinjectiveenvelope}
Assume  that $\V_0$ is Grothendieck (that is, $\V$ is a Grothendieck cosmos). Every object in $\V_0$ has a geometrically pure injective envelope, that is, an envelope w.r.t.~the class $\PureInj{\V_0}$.
\end{thm}

\begin{proof}
Let $\mathbb{A}$ be the class of geometrically pure monomorphisms and $\mathbb{J} = \PureInj{\V_0}$ be the class of geometrically pure injective objects in $\V_0$. The following conditions hold:
\begin{rqm}
\item An object $J \in \V_0$ belongs to $\mathbb{J}$ if and only if $\V_0(Y,J) \to \V_0(X,J) \to 0$ is exact in $\Ab$ for every $X \to Y$ in $\mathbb{A}$.
\item A morphism $X \to Y$ in $\V_0$ belongs to $\mathbb{A}$ if and only if $\V_0(Y,J) \to \V_0(X,J) \to 0$ is exact in $\Ab$ for every $J \in \mathbb{J}$.
\item Every object in $\V_0$ has a $\mathbb{J}$-preenvelope.
\end{rqm}
Indeed, the ``only if'' part of \rqmlbl{1} holds by definition of geometrically pure injective objects. For the ``if'' part, take by \prpref{dual2} a morphism $J \to J'$ in $\mathbb{A}$ with $J' \in \mathbb{J}$. By~assump\-tion, $\V_0(J',J) \to \V_0(J,J) \to 0$ is exact, so $\mathrm{id}_J$ has a left-inverse $J' \to J$. Thus $J$ is a direct summand in $J' \in \mathbb{J}$ and it follows that $J \in \mathbb{J}$. The ``only if'' part of \rqmlbl{2} holds by definition of geometrically pure injective objects. For the ``if'' part, let $f \colon X \to Y$ be any morphism~in~$\V_0$. For every $V \in \V_0$ one has $[V,E] \in \mathbb{J}$ by \lemref{dual}, so $\V_0(f,[V,E])$ is surjective by assump\-tion. As in the proof of \lemref[]{dual}, this menas that 
$f \otimes V$ is monic, and hence $f$ is in $\mathbb{A}$. Condition \rqmlbl{3} holds by \prpref{dual2}.

These arguments show that $(\mathbb{A},\mathbb{J})$ is an \emph{injective structure} in the sense of \cite[Dfn.~6.6.2]{rha}. Even though the definitions and results (with proofs) about such structures found in \cite{rha} are formulated for the category of modules over a ring, they carry over to any Grothendieck cosmos. Since the injective structure $(\mathbb{A},\mathbb{J})$ is \emph{determined} by the class $\mathcal{G}:=\V_0$ in the sense of \cite[Dfn.~6.6.3]{rha}, the desired conclusion follows from \cite[Thm.~6.6.4(1)]{rha}.
\end{proof}

It is well-known that if $\mathcal{K}$ is a small ordinary category, then the category of func\-tors $\mathcal{K} \to \Ab$ is abelian, and even Grothendieck. In \cite[Thm.~4.2]{AG16} it is shown that if $\mathcal{K}$ is a small $\V$-category, then the ordinary category $[\mathcal{K},\V]_0$ of \textsl{$\V$-functors} $\mathcal{K} \to \V$ is abelian too, and even Grothendieck if $\V$ is\footnote{\,Note that in \emph{loc.~cit.} the symbol $[\mathcal{K},\V]$ is used for the \textsl{ordinary} category of $\V$-functors $\mathcal{K} \to \V$ (but we use the symbol $[\mathcal{K},\V]_0$) whereas the \textsl{$\V$-category} of such functors is denoted by $\mathcal{F}(\mathcal{K})$ (but we use the symbol $[\mathcal{K},\V]$).}. Moreover, (co)limits, in particular, (co)kernels, in the category $[\mathcal{K},\V]_0$ are formed objectwise. Below we construct a certain exact structure on $[\mathcal{K},\V]_0$. 

\begin{lem} 
\label{lem:3T} 
Let $\mathcal{K}$ be a small $\V$-category and let $0 \to F' \to F \to F'' \to 0$ be an exact sequence in the abelian category $[\mathcal{K},\V]_0$. For every $\V$-functor $G \colon \mathcal{K}^\mathrm{op} \to \V$ the sequence $G \star F' \to G \star F \to G \star F'' \to 0$ is exact in $\V_0$.
\end{lem}

\begin{proof}
%For any $K \in \mathcal{K}$ the sequence
%$G(K) \otimes F'(K) \to 
%  G(K) \otimes F(K) \to 
%  G(K) \otimes F''(K) \to 0$ is exact in $\V_0$ by right exactness of the (left %adjoint) functor $G(K) \otimes - \colon \V_0 \to \V_0$. Since $G \star F$ can be %computed as the coend \smash{$\int^{K} G(K) \otimes F(K)$}, and similarly for %$F'$ and $F''$, right exactness of (ordinary) colimits in $\V_0$ yields the %result.
It follows immediately from the fact that $G \star F \cong F \star G$, see \cite[eq.~(3.9)]{Kelly}, and from the axiom (\ref{eq:eq-wcolim}) in the definition of weighted colimits, that $(- \star G, [G,-])$ is an adjoint pair. This implies that the functor $G \star - \cong - \star G$ is right exact.
\end{proof}

\begin{prp}
\label{prp:star}
Let $\mathcal{K}$ be a small $\V$-category and let $0 \to F' \to F \to F'' \to 0$ be an exact sequence in $[\mathcal{K},\V]_0$. The following conditions are equivalent:
\begin{eqc}
\item $0 \to G \star F' \to G \star F \to G \star F'' \to 0$ is an exact sequence in $\V_0$ for every $G \in [\mathcal{K}^\mathrm{op},\V]$.
\item $0 \to [F'',E] \to [F,E] \to [F',E] \to 0$ is a split short exact sequence in $[\mathcal{K}^\mathrm{op},\V]_0$.
\end{eqc}
\end{prp}

\begin{proof}
Let $\mathbb{S}$ be the given exact sequence. By the definition of weighted colimits, see \eqref{eq-wcolim}, there is an isomorphism $[G \star \mathbb{S}, E] \cong [\mathcal{K}^{\mathrm{op}},\V](G, [\mathbb{S},E])$ of sequences in $\V_0$ and thus an induced isomorphism of sequences in $\Ab$,
\begin{equation}
  \label{eq:eq-wcolim-app}
\V_0(G \star \mathbb{S}, E) \,\cong\, [\mathcal{K}^{\mathrm{op}},\V]_0(G, [\mathbb{S},E])\;.
\end{equation}
As $E$ is an injective cogenerator in $\V_0$, condition \eqclbl{i} holds if and only if the left-hand side in \eqref{eq-wcolim-app} is exact for every $G \in [\mathcal{K}^\mathrm{op},\V]$. Evidently, \eqclbl{ii} holds if and only if right-hand side in \eqref{eq-wcolim-app} is exact for every $G \in [\mathcal{K}^\mathrm{op},\V]$. Hence \eqclbl{i} and \eqclbl{ii} are equivalent.
\end{proof}

\begin{dfn}
\label{dfn:star-exact-structure} 
Let $\mathscr{E}_{\star}$ denote the class of all short exact sequences in $[\mathcal{K},\V]_0$ that satisfy the equivalent conditions in \prpref{star}. We call $\mathscr{E}_{\star}$ the \emph{$\star$-pure exact structure} on $[\mathcal{K},\V]_0$ (see the next result). Sequences in $\mathscr{E}_{\star}$ are called \textit{$\star$-pure (short) exact sequences}.
\end{dfn}

\begin{prp}
\label{prp:prp-star} 
The pair $([\mathcal{K},\V]_0,\mathscr{E}_{\star})$ is an exact category.
\end{prp}

\begin{proof}
  Apply \prpref[Prop.~]{construction-of-exact-structure} with $\mathcal{C} = [\mathcal{K},\V]_0$ and $\mathfrak{T}$ the class of functors $G \star - \colon [\mathcal{K},\V]_0 \to \V_0$ where $G \in [\mathcal{K}^\mathrm{op},\V]$. Note that every functor $G \star -$ is right exact by \lemref{3T}.
\end{proof}  

Recall that a left $R$-module $M$ is \emph{absolutely pure} (or \emph{FP-injective}) if it is a pure submodule of every $R$-module that contains it; see \cite[Dfn.~A.17]{JL}. Equivalently, every short exact sequence $0 \to M \to K \to K' \to 0$ is pure exact, that is, $0 \to X \otimes M \to X \otimes K \to X \otimes K' \to 0$ is exact for every right $R$-module $X$. The definition of absolutely pure $\V$-functors $\mathcal{K} \to \V$ given below is completely analogous to this.

\begin{dfn}
  \label{dfn:abspure}
  Let $\mathcal{K}$ be a small $\V$-category and let $H$ be an object in $[\mathcal{K},\V]_0$. Recall that 
\begin{equation*}
  \textnormal{$H$ is injective}
  \ \ \iff \ \ 
  \left\{\!\!
    \begin{array}{l}
      \textnormal{Every exact sequence $0 \to H \to F \to F' \to 0$}
      \\
      \textnormal{in the abelian category $[\mathcal{K},\V]_0$ is split exact\,.}
    \end{array}
  \right.
\end{equation*}
Inspired by the remarks above, we define:
\begin{equation*}
  \textnormal{$H$ is \emph{absolutely pure}}
  \ \ \iff \ \ 
  \left\{\!\!
    \begin{array}{l}
      \textnormal{Every exact sequence $0 \to H \to F \to F' \to 0$}
      \\
      \textnormal{in the abelian category $[\mathcal{K},\V]_0$ is $\star$-pure exact\,.}
    \end{array}
  \right.
\end{equation*}
We also set:
\begin{align*}
  \Inj{[\mathcal{K},\V]_0} &\,=\, \{H \in [\mathcal{K},\V]_0 \ | \ \textnormal{$H$ is injective} \}
  \quad \textnormal{and}
  \\
  \AbsPure{[\mathcal{K},\V]_0} &\,=\, \{H \in [\mathcal{K},\V]_0 \ | \ \textnormal{$H$ is absolutely pure} \}\;.  
\end{align*}  
\end{dfn}

These categories will appear in \thmref{main3}, the final result of the paper.

\section{The tensor embedding for an abelian cosmos}
\label{sec:tensor-embedding}

We establish some general properties of the tensor embedding defined in \dfnref[]{Def-Psi} below. The main result is \thmref{main}, which shows that the tensor embedding identifies the geometrically pure exact category $(\V_0,\mathscr{E}_\otimes)$ from \prpref{geometricalexactst} with a certain exact subcategory of $([\A,\V]_0,\mathscr{E}_{\star})$ from \prpref{prp-star}.

\begin{stp}
\label{stp:setup-abelian-cosmos-2} The setup for this section is the same as in \stpref[]{setup-abelian-cosmos-1}, i.e~$(\V,\otimes,I,[-,-])$ is an abelian\footnote{\,Note that the abelianness of $\V_0$ is not used, neither is it important for, part \prtlbl{a} in \thmref{main}.} cosmos with an injective cogenerator $E$. 
\end{stp}

\begin{dfn}
\label{dfn:Def-Psi} 
Recall from \cite[\S1.6]{Kelly} that $\otimes$ is a $\V$-functor\footnote{\,Note that in \emph{loc.~cit.} the symbol ``Ten'' is used for this $\V$-functor whereas ``$\otimes$'' is reserved for the ordinary functor $\V_0\times \V_0 \to \V_0$, however, we shall abuse notation and use the latter symbol for both functors.} \mbox{$\V \pmb{\otimes} \V \to \V$}. For a small full $\V$-subcategory $\A$ of $\V$, restriction yields a $\V$-functor \mbox{$\otimes \colon \V \pmb{\otimes} \A \to \V$}. Via the isomorphism 
\begin{equation*}
\V\text{-CAT}(\V \pmb{\otimes} \A,\V) \,\cong\, \V\text{-CAT}(\V,[\A,\V])
\end{equation*}
from \cite[\S2.3 eq.~(2.20)]{Kelly}, the latter $\V$-functor corresponds to the $\V$-functor 
\begin{equation*}
\upTheta \colon \V \longrightarrow [\A,\V]
\qquad \text{given by} \qquad 
X \longmapsto (X\otimes-)|_{\A} \colon \A \to \V\;.
\end{equation*}
We refer to this $\V$-functor as the \emph{tensor embedding}. Note that it induces an additive functor $\upTheta_0 \colon \V_0 \to [\A,\V]_0$ of the underlying abelian categories.
\end{dfn}

\begin{rmk}
If $I \in \A$, then the $\V$-functor $\A(I,-) = [I,-]$ exists and it is clearly naturally isomorphic to the inclusion $\V$-functor, $\mathrm{inc} \colon \A \to \V$. Thus, in the notation of \ref{tensored-cotensored} one has 
\begin{equation}
  \label{eq:eq-Psi-computation}
  \upTheta(X) \,=\, (X\otimes-)|_{\A} \,=\, X \otimes \mathrm{inc} \,\cong\, X \otimes \A(I,-)\;.
\end{equation}
\end{rmk}

For the next result, recall the notions of geometrically pure exact sequences and $\star$-pure exact sequences from \dfnref[Definitions~]{geo-exact-structure} and \dfnref[]{star-exact-structure}.

\begin{lem}
\label{lem:i-ii} 
Let $\A$ be a small full $\V$-subcategory of $\mspace{1.5mu}\V$ and let $\mathbb{S}$ be a short exact sequence in $\V_0$. The following two conditions are equivalent:
\begin{eqc}
\item $\mathbb{S} \otimes A$ is a short exact sequence in $\V_0$ for every $A \in \A$.
\item $\upTheta_0(\mathbb{S})$ is a short exact sequence in $[\A,\V]_0$.
\end{eqc}
If $I$ belongs to $\A$, then the following two conditions are equivalent:
\begin{eqc}
\item[\textnormal{($i^\prime$)}] $\mathbb{S}$ is a geometrically pure exact sequence in $\V_0$.
\item[\textnormal{($ii^\prime$)}] $\upTheta_0(\mathbb{S})$ is a $\star$-pure exact sequence in $[\A,\V]_0$.
\end{eqc}
\end{lem}

\begin{proof}
  The equivalence \eqclbl{i}\,$\Leftrightarrow$\,\eqclbl{ii} is evident from the definitions. Now assume that $I \in \A$. For every $\V$-functor $G \colon \A^\mathrm{op} \to \V$ there is an equivalence of endofunctors on $\V_0$,
\begin{equation}
  \label{eq:eq-star-tensor}
  G \mspace{1mu}\star\mspace{1mu} \upTheta_0(-) \,\cong\,  - \otimes G(I)\;.
\end{equation}
Indeed, for $X \in \V$ one has the next isomorphisms, where the $1^\mathrm{st}$ is by \eqref{eq-Psi-computation}, the $2^\mathrm{nd}$ follows as 
the $\V$-functor $X \otimes\, ? \colon \V \to \V$ preserves weighted colimits (this follows from e.g.~\cite[Prop. 6.6.12]{B2}), and the $3^\mathrm{rd}$ is by \cite[eq.~(3.10)]{Kelly}:
\begin{equation*}
  G \mspace{1mu}\star\mspace{1mu} \upTheta(X) \,\cong\, 
  G \mspace{1mu}\star\mspace{1mu} (X \otimes \A(I,-)) \,\cong\, 
  X \otimes (G \mspace{1mu}\star\mspace{1mu} \A(I,-)) \,\cong\,
  X \otimes G(I)\;. 
\end{equation*}  
It is clear from \eqref{eq-star-tensor} that \eqclbl{i$^\prime$} implies \eqclbl{ii$^\prime$}. Conversely, assume \eqclbl{ii$^\prime$} and let $V \in \V$ be given. As $G=[-,V]$ is a $\V$-functor $\A^\mathrm{op} \to \V$, the sequence \mbox{$[-,V] \star \upTheta (\mathbb{S})$} is exact by assumption. Another application of \eqref{eq-star-tensor} shows that $\mathbb{S} \otimes [I,V] \cong \mathbb{S} \otimes V$ is exact, so \eqclbl{i$^\prime$} holds.
\end{proof}

Recall that the \emph{essential image} of a $\V$-functor $T \colon \mathcal{C} \to \mathcal{D}$, denoted by $\operatorname{Ess.Im}T$, is just the essential image of the underlying ordinary functor $T_0 \colon \mathcal{C}_0 \to \mathcal{D}_0$. Thus $\operatorname{Ess.Im}T$ is the collection of all objects $D \in \mathcal{D}$ such that $D \cong T(C)$ in $\mathcal{D}_0$ for some object $C \in \mathcal{C}$. We may consider $\operatorname{Ess.Im}T$ as a full $\V$-subcategory of $\mathcal{D}$ or as a full subcategory of $\mathcal{D}_0$.

\begin{lem} 
\label{lem:induced-exact} 
Let $\A$ be a small full $\V$-subcategory of $\mspace{1.5mu}\V$ with $I \in \A$. For any short exact sequence $0 \to F' \to F \to F'' \to 0$ in the abelian category $[\A,\V]_0$ one has:
\begin{equation*}
  F',F'' \in \operatorname{Ess.Im}\upTheta \quad \Longrightarrow \quad
  F \in \operatorname{Ess.Im}\upTheta\;.
\end{equation*}
Consequently, $\operatorname{Ess.Im}\upTheta$ is an extension-closed subcategory of both of the exact categories
\begin{equation*}
  ([\A,\V]_0,\mathscr{E}_\mathrm{ab})
  \qquad \textnormal{and} \qquad
  ([\A,\V]_0,\mathscr{E}_\star)\;,
\end{equation*}
where $\mathscr{E}_\mathrm{ab}$ is the exact strucure induced by the abelian structure, and $\mathscr{E}_\star$ is the (coarser) exact structure from \prpref{prp-star}. It follows that the sequences in $\operatorname{Ess.Im}\upTheta$ which are exact, respectively, $\star$-pure exact, in $[\A,\V]_0$ form an exact structure on $\operatorname{Ess.Im}\upTheta$, which we denote by $\mathscr{E}_\mathrm{ab}|_{\operatorname{Ess.Im}\upTheta}$, respectively, $\mathscr{E}_\star|_{\operatorname{Ess.Im}\upTheta}$. In this way, we obtain exact categories:
\begin{equation*}
  (\operatorname{Ess.Im}\upTheta,\mspace{1mu} \mathscr{E}_\mathrm{ab}|_{\operatorname{Ess.Im}\upTheta})
  \qquad \textnormal{and} \qquad
  (\operatorname{Ess.Im}\upTheta,\mspace{1mu} \mathscr{E}_\star|_{\operatorname{Ess.Im}\upTheta})\;.
\end{equation*}
\end{lem}

\begin{proof}
  For all objects $X \in \V$ and $F \in [\A,\V]$ there are isomorphisms:
\begin{equation}
  \label{eq:adj}
  [\A,\V](\upTheta(X),F) \,\cong\,
  [\A,\V](X \otimes \A(I,-),F) \,\cong\, 
  [X, [\A,\V](\A(I,-),F)] \,\cong\,
  [X,F(I)]
\end{equation} 
which follow from \eqref{eq-Psi-computation}, \eqref{eq-tensored-cotensored}, and the strong Yoneda Lemma \cite[\S2.4 eq.~(2.31)]{Kelly}. Consequently there is also an isomorphism
$[\A,\V]_0(\upTheta(X),F) \cong \V_0(X,F(I))$, which shows that there is a pair of adjoint functors $(\upTheta_0,(\mathrm{Ev}_I)_0)$ where $\mathrm{Ev}_I$ is the $\V$-functor given by evaluation at the unit object $I$ (see \cite[\S2.2]{Kelly}):
\begin{equation*}
\xymatrix@C=5pc{
  \V_0
  \ar@<0.6ex>[r]^-{\upTheta_0}
  &
  [\A,\V]_0\;.\mspace{-8mu}
  \ar@<0.6ex>[l]^-{(\mathrm{Ev}_I)_0}
}
\end{equation*}
Write $\theta$ for the counit of this adjunction; thus for $F \in [\A,\V]_0$ we have the $\V$-natural transformation $\theta_F \colon \upTheta_0(F(I)) \to F$. Clearly, $\theta_F$ is an isomorphism if and only if $F \in \operatorname{Ess.Im}\upTheta$.

Now, let $0 \to F' \to F \to F'' \to 0$ be an exact sequence in $[\A,\V]_0$. It induces an exact sequence $0 \to F'(I) \to F(I) \to F''(I) \to 0$ in $\V_0$ and hence the exact sequence in the upper row of the next commutative diagram (by right exactness of the tensor product),
\begin{equation*}
  \xymatrix{
    {} & 
    \upTheta_0(F'(I)) \ar[r] \ar[d]^-{\theta_{F'}} &
    \upTheta_0(F(I)) \ar[r] \ar[d]^-{\theta_{F}} &
    \upTheta_0(F''(I)) \ar[r] \ar[d]^-{\theta_{F''}} &
    0
    \\
    0 \ar[r] & F' \ar[r] & F \ar[r] & F'' \ar[r] & 0\;.\mspace{-8mu}
  }
\end{equation*}
If $F'$ and $F''$ are in $\operatorname{Ess.Im}\upTheta$, then $\theta_{F'}$ and $\theta_{F''}$ are  isomorphisms in $[\A,\V]_0$; whence $\theta_F$ is an isomorphism by the Five Lemma, so $F$ belongs to $\operatorname{Ess.Im}\upTheta$.

The last assertion follows directly from \cite[Lem.~10.20]{Buhler} (also note that $\operatorname{Ess.Im}\upTheta$ is an additive subcategory of $[\A,\V]_0$ since the functor $\upTheta_0$ is additive).
\end{proof}

\begin{thm}
\label{thm:main} 
Let $\V$ be as in \stpref{setup-abelian-cosmos-2} and let $\A$ be any small full $\V$-subcategory of $\V$. The tensor embedding 
\begin{equation*}
\upTheta \colon \V \longrightarrow [\A,\V]
\qquad \text{given by} \qquad 
X \longmapsto (X\otimes-)|_{\A}\;.
\end{equation*}
from \dfnref{Def-Psi} is cocontinuous, that is, it preserves all small weighted colimits. If $I$ belongs to $\A$, then $\upTheta$ is a fully faithful and it induces two equivalences:
\begin{prt}
\item An equivalence of\, $\V$-categories, $\upTheta \colon \V \stackrel{\simeq}{\longrightarrow} \operatorname{Ess.Im}\upTheta$.

\item An equivalence of exact categories, $\upTheta_0 \colon (\V_0,\mathscr{E}_\otimes) \stackrel{\simeq}{\longrightarrow} (\operatorname{Ess.Im}\upTheta,\mathscr{E}_\star|_{\operatorname{Ess.Im}\upTheta})$.
\end{prt}
\end{thm}

\begin{proof}
First note that the $\V$-categories $\V$ and $[\A,\V]$ are bicomplete, that is, they have small weighted limits and colimits. This follows from \cite[\S3.1 and \S3.3]{Kelly} as the ordinary category $\V_0$ is assumed to be bicomplete. We now show that $\upTheta$ is cocontinuous. 

Let $\mathcal{K}$ be a small $\V$-category and let $G \colon \mathcal{K}^{\mathrm{op}} \to \V$ and $F \colon \mathcal{K} \to \V$ be $\V$-functors. We must show that $G \star (\upTheta \circ F) \cong \upTheta(G \star F)$, where the weighted colimit on the left-hand side is computed in $[\A,\V]$ and the one on the right-hand side in $\V$. Via the isomorphism 
\begin{equation*}
  \V\text{-CAT}(\mathcal{K},[\A,\V]) \,\cong\, \V\text{-CAT}(\mathcal{K} \pmb{\otimes} \A,\V)
\end{equation*}  
from \cite[\S2.3 eq.~(2.20)]{Kelly}, the $\V$-functor $\upTheta \circ F \colon \mathcal{K} \to [\A,\V]$ corresponds to the $\V$-functor $P \colon \mathcal{K} \pmb{\otimes} \A \to \V$ given by $P(K,A) = F(K) \otimes A$. So $P(-,A)$ is $F \otimes A$ in the notation of \ref{tensored-cotensored}. In the computation below, the $1^{\mathrm{st}}$ isomorphism holds as weighted colimits in $[\A,\V]$ are computed objectwise, see \cite[\S3.3]{Kelly} (more precisely, we use the weighted colimit counterpart of eq.~(3.16) in \emph{loc.~cit.}); the $2^{\mathrm{nd}}$ and $3^{\mathrm{rd}}$ isomorphisms hold by the definitions of $P$ and $\upTheta$; the $4^{\mathrm{th}}$ isomorphism holds as the $\V$-functor $- \otimes A \colon \V \to \V$ preserves weighted~colimits:
\begin{equation*}
  (G \star (\upTheta \circ F))(A) 
  \,\cong\,
  G \star P(-,A)
  \,\cong\, 
  G \star (F \otimes A)
  \,\cong\, 
  (G \star F) \otimes A 
  \,\cong\, 
  \upTheta(G \star F)(A)\;.
\end{equation*}
Consequently, $G \star (\upTheta \circ F) \cong \upTheta(G \star F)$, as claimed. 

Now assume that $I \in \A$. As $\upTheta$ is a $\V$-functor it comes equipped with a natural morphism,
\begin{equation*}
  \upTheta_{XY} \colon [X,Y] \longrightarrow [\A,\V](\upTheta(X),\upTheta(Y))\;,
\end{equation*}
for every pair of objects $X,Y \in \V$. The claim is that $\upTheta_{XY}$ is an isomorphism in $\V$. However, the morphism $\upTheta_{XY}$ is precisely the following composite, where the second isomorphism comes from \eqref{adj} with $F=\upTheta(Y)$,
\begin{equation*}
  [X,Y] \,\cong\, [X,\upTheta(Y)(I)] \,\cong\, [\A,\V](\upTheta(X),\upTheta(Y))\;.
\end{equation*} 

\proofoftag{a} The asserted $\V$-equivalence is a formal consequence of fact that $\upTheta$ is fully faithful; see \cite[\S1.11 p.\,24]{Kelly}. 

\proofoftag{b} By part \prtlbl{a} we have an equivalence of additive categories \smash{$\upTheta_0 \colon \V_0 \stackrel{\simeq}{\longrightarrow} \operatorname{Ess.Im}\upTheta$}. The assertion about exact categories now follows from the equivalence \eqclbl{i$^\prime$}\,$\Leftrightarrow$\,\eqclbl{ii$^\prime$} in \lemref{i-ii} and from \lemref{induced-exact}.
\end{proof}

We end this section with two results that show how to construct a (co)generating set of objects in the category $[\mathcal{K},\V]_0$ of $\V$-functors from a (co)generating set of objects in $\V_0$. This will be used in the proof of \prpref{injective}.

\begin{lem}
\label{lem:c1}
Let $\mathcal{K}$ be a small $\V$-category. The following hold.
\begin{prt}
\item If $\mathcal{S}$ is a cogenerating set of objects in $\V_0$, then \mbox{$\{\,[\mathcal{K}(-,K),S]\,\}_{K\in \mathcal{K},\,S\in \mathcal{S}}$} is a cogenerating set of objects in $[\mathcal{K},\V]_0$.
\item If $S \in \V_0$ is injective, then $[\mathcal{K}(-,K),S]$ is injective in $[\mathcal{K},\V]_0$ for every $K \in \mathcal{K}$. 
\end{prt}
\end{lem}

\begin{proof}
For $K$ in $\mathcal{K}$, the functor  $[\mathcal{K},\V]_0 \to \V_0$ given by \mbox{$? \,\mapsto\ \mathcal{K}(-,K) \,\star\, ?$} is just the evaluation functor $\mathrm{Ev}_K(?)$ at $K$, see \cite[\S3.1 eq.~(3.10)]{Kelly}. This fact, the defining property \eqref{eq-wcolim} 
of weighted colimits, and \cite[\S3.1 eq.~(3.9)]{Kelly} yield isomorphisms in $\V$,
\begin{equation*}
[\mathcal{K},\V](?\,,[\mathcal{K}(-,K),S])
\,\cong\,
[\mathcal{K}(-,K) \,\star\, ?\, ,S]
\,\cong\,
[\mathrm{Ev}_K(?),S]\;.
\end{equation*}
In particular, $[\mathcal{K},\V]_0(?,[\mathcal{K}(-,K),S]) \cong \V_0(\mathrm{Ev}_K(?),S)$, which yields both assertions.
\end{proof}

\lemref{c1} has the next dual of which part \prtlbl{a} can also be found in \cite[Thm.~4.2]{AG16}.

\begin{lem}
Let $\mathcal{K}$ be a small $\V$-category. The following hold.
\begin{prt}
\item If $\mathcal{S}$ is a generating set of objects in $\V_0$, then $\{S \otimes  \mathcal{K}(K,-)\}_{K \in \mathcal{K},\,S \in \mathcal{S}}$ is a generating set of objects in $[\mathcal{K},\V]_0$.
\item If $S \in \V_0$ is projective, then $S \otimes  \mathcal{K}(K,-)$ is projective in $[\mathcal{K},\V]$ for every $K \in \mathcal{K}$.
\end{prt}
\end{lem}

\begin{proof}
For any objects $S \in \V$ and $K \in \mathcal{K}$, the first iso\-mor\-phism below follows from \eqref{eq-tensored-cotensored} and the second follows from the strong Yoneda Lemma \cite[\S2.4 eq.~(2.31)]{Kelly}:
\begin{equation*}
[\mathcal{K},\V](S \otimes  \mathcal{K}(K,-),?)
\,\cong\,
[S,[\mathcal{K},\V](\mathcal{K}(K,-),?)]
\,\cong\,
[S,\mathrm{Ev}_K(?)]\;.
\end{equation*}
In particular, $[\mathcal{K},\V]_0(S \otimes  \mathcal{K}(K,-),?) \cong \V_0(S,\mathrm{Ev}_K(?))$, which yields both assertions.
\end{proof}

\section{The tensor embedding for a Grothendieck cosmos}
\label{sec:Grothendieck-cosmos}

In this section, we work with \stpref{Setup-cosmos-Grothendieck} below and we consider the tensor embedding 
\begin{equation}
  \label{eq:tenemb}
  \upTheta \colon \V \longrightarrow [\lPres{\V},\V]
\end{equation}
from \dfnref[Definition~]{Def-Psi} in the special case where $\A = \lPres{\V}$. The goal is to strengthen and make \thmref{main} more explicit in this situation; this is achieved in \thmref{main2} below. Note that the examples found in \exaref[]{Ch} and \exaref[]{QcohX} all satisfy the following setup.

\begin{stp}
\label{stp:Setup-cosmos-Grothendieck} 
In this section, $(\V,\otimes,I,[-,-])$ is a cosmos and $\V$ is a Grothendieck category. 
\begin{itemlist}
\item We fix a regular cardinal $\lambda$ such that $\V$ is a locally $\lambda$-presentable base\footnote{\,Thus the blanket setup at the end of the Introduction in \cite{BQR98} is satisfied, and we can apply the theory herein.} in the sense of \ref{lpb}; such a choice is possible by \prpref{Grothendieck-lpb} below. 
\item We let $\lPres{\V}$ be the (small) collection of all $\lambda$-presentable objects in $\V$.
\item We fix an injective cogenerator $E$ in $\V_0$; existence is guaranteed by \cite[Thm.~9.6.3]{KS06}.
\end{itemlist}
\end{stp}

\begin{prp}
\label{prp:Grothendieck-lpb} 
There is a regular cardinal $\lambda$ for which $\V$ is a locally $\lambda$-presentable base.
\end{prp}

\begin{proof}
  As $\V_0$ is Grothendieck, it follows from \cite[Prop.~3.10]{Beke} that it is a locally $\gamma$-present\-able category for some regular cardinal $\gamma$. Note that for every regular cardinal $\lambda \geqslant \gamma$, the category $\V_0$ is also locally $\lambda$-presentable by \cite[Remark after Thm.~1.20]{AR}, so condition \rqmlbl{1} in \ref{lpb} holds for all such $\lambda$. Let $\mathcal{S}$ be a set of representatives for the isomorphism classes of $\gamma$-presentable objects in $\V_0$. Let $\mathcal{S}'$ be the set consisting of the unit object $I$ and all finite tensor products of objects in $\mathcal{S}$. As every object in $\V_0$ is presentable (that is, $\mu$-presentable for some regular cardinal $\mu$), see again \cite[Remark after Thm.~1.20]{AR}, and since $\mathcal{S}'$ is set, there exists some regular cardinal $\lambda \geqslant \gamma$ such that every object in $\mathcal{S}'$ is $\lambda$-presentable. In particular, condition \rqmlbl{2} in \ref{lpb} holds. If neccessary we can replace $\lambda$ with its successor $\lambda^+$ (every successor cardinal is regular) and thus by \cite[Exa.~2.13(2)]{AR} assume that $\gamma$ is \emph{sharply smaller} than $\lambda$ (in symbols: $\gamma \vartriangleleft \lambda$) in the sense of \cite[Dfn.~2.12]{AR}. This will play a role in the following argument, which shows that condition \rqmlbl{3} in \ref{lpb} holds.
  
    Let $X$ and $Y$ be $\lambda$-presentable objects in $\V_0$. As the category $\V_0$ is locally $\gamma$-presentable (and hence also $\gamma$-accessible) and $\lambda \vartriangleright \gamma$, it follows from    \cite[Rem.~2.15]{AR} that $X$ and $Y$ are direct sum\-mands of a $\lambda$-small directed colimit of $\gamma$-presentable objects in $\V_0$, i.e.~we have
\begin{equation*}
  X \oplus X' \,\cong\, \colim_{p \in P}X_p
  \qquad \textnormal{and} \qquad
  Y \oplus Y' \,\cong\, \colim_{q \in Q}Y_q
\end{equation*}  
where $X_p, Y_q \in \mathcal{S}$ and $|P|,|Q| < \lambda$. As $\otimes$ preserves all colimits, one has
\begin{equation}
\label{eq:double-colim}
(X \oplus X') \otimes (Y \oplus Y') \,\cong\, \colim_{(p,q) \in P \times Q}\, X_p \otimes Y_q\;. 
\end{equation}
By construction, each $X_p \otimes Y_q$ is in $\mathcal{S'}$, so it is a $\lambda$-presentable object. Furthermore, the category $P \times Q$ is $\lambda$-small. Hence \cite[Prop.~1.16]{AR} and \eqref{double-colim} imply that $(X \oplus X') \otimes (Y \oplus Y')$ is $\lambda$-presentable. 
Since $X \otimes Y$ is a direct summand of $(X \oplus X') \otimes (Y \oplus Y')$, the object $X \otimes Y$ is $\lambda$-presentable too by \cite[Remark after Prop.~1.16]{AR}.\footnote{\,According to \cite[Rem.~1.30(2)]{AR} it follows from \cite{MP} that if $\lambda$ is any regular cardinal $\geqslant \gamma$, then every $\lambda$-pre\-sent\-able object is a $\lambda$-small colimit of $\gamma$-presentable objects. If this is true, then a couple of simplifications can be made in the proof of \prpref{Grothendieck-lpb}. Indeed, in this case we would not have to worry about $\gamma$ being \emph{sharply smaller} than $\lambda$, and we could simply take both $X'$ and $Y'$ to be zero. However, as there seems to be some doubt about the correctness of the claim in \cite[Rem.~1.30(2)]{AR} (\url{https://mathoverflow.net/questions/325278/mu-presentable-object-as-mu-small-colimit-of-lambda-presentable-objects}), we have chosen to give a (slightly more complicated) proof of \prpref{Grothendieck-lpb} based on 
\cite[Rem.~2.15]{AR} instead.}
\end{proof}

Under \stpref{Setup-cosmos-Grothendieck} one can improve some of the statements about purity from \secref[Sections~]{purity} and \secref[]{tensor-embedding}. This will be our first goal.
	
\begin{lem}
\label{lem:gps11} 
For a short exact sequence $\mathbb{S}$ in $\V_0$, the following conditions are equivalent:
\begin{eqc}
\item $\mathbb{S} \otimes C$ is a short exact sequence in $\V_0$ for every $C \in \lPres{\V}$.
\item $\V_0(\mathbb{S},[C,E])$ is a short exact sequence in $\Ab$ for every $C \in \lPres{\V}$.
\item $\mathbb{S}$ is geometrically pure exact sequence in $\V_0$.
\end{eqc}
\end{lem}

\begin{proof}
The equivalence \eqclbl{i}\,$\Leftrightarrow$\,\eqclbl{ii} follows as $\V_0(\mathbb{S} \otimes C,E) \cong \V_0(\mathbb{S},[C,E])$ and $E$ is an injective cogenerator in $\V_0$. Evidently, \eqclbl{iii}\,$\Rightarrow$\,\eqclbl{i}. Finally, assume that \eqclbl{i} holds. Every object $V \in \V$ is a directed colimit of $\lambda$-presentable objects, say, $V \cong \colim_{q \in Q}C_q$. The sequence $\mathbb{S} \otimes V \cong \colim_{q \in Q}(\mathbb{S} \otimes C_q)$ is exact as each $\mathbb{S} \otimes C_q$ is exact (by assumption) and any directed colimit of exact sequences is again exact because $\V_0$ is Grothendieck. So \eqclbl{iii}~holds.
\end{proof}

\begin{cor}
  \label{cor:ab-equals-star}
  Consider the tensor embedding \textnormal{\eqref{tenemb}}. The exact structures $\mathscr{E}_\mathrm{ab}$ and $\mathscr{E}_\star$ on the category $[\lPres{\V},\V]_0$ from \lemref{induced-exact} agree on the subcategory $\operatorname{Ess.Im}\upTheta$, that is,
\begin{equation*}
  \mathscr{E}_\mathrm{ab}|_{\operatorname{Ess.Im}\upTheta}
  \,=\,
  \mathscr{E}_\star|_{\operatorname{Ess.Im}\upTheta}\;.
\end{equation*}  
\end{cor}

\begin{proof}
  Every short exact sequence $0 \to F' \to F \to F'' \to 0$ in $\operatorname{Ess.Im}\upTheta$ has the form (up to isomorphism) $\upTheta_0(\mathbb{S})$ for some short exact sequence $\mathbb{S} = 0 \to X' \to X \to X'' \to 0$~in~$\V_0$. \lemref{gps11} shows that conditions \eqclbl{i} and \eqclbl{i$^\prime$} in \lemref{i-ii} (with $\A = \lPres{\V}$) are equivalent. Thus \eqclbl{ii} and \eqclbl{ii$^\prime$} in \lemref{i-ii} are equivalent too, as asserted.
\end{proof}

We know from \prpref{dual2} that the exact category $(\V_0,\mathscr{E}_\otimes)$ has enough relative injectives. An alternative demonstration of this fact is contained in the next proof. 

\begin{prp}
\label{prp:gpi-product}
An object $X\in \V_0$ is geometrically pure injective if and only if it is a direct summand of an object $\prod_{q \in Q}\mspace{1mu}[B_q,E]$ for some family $\{B_q\}_{q \in Q}$ of $\lambda$-presentable objects.
\end{prp}

\begin{proof}
  The ``if'' part is clear since each object $[B_q,E]$ is geometrically pure injective by \lemref{dual}. Conversely, let $X$ be any object in $\V_0$. Choose a set $\mathcal{C}$ of representatives for the isomorphism classes of $\lambda$-presentable objects in $\V_0$ and consider the canonical morphism
\begin{equation*}
  \alpha \colon X \longrightarrow \textstyle\prod_{C \in \mathcal{C}}\mspace{1mu}[C,E]^{J_C}
  \qquad \textnormal{where} \qquad J_C = \V_0(X,[C,E])\;.
\end{equation*}
We will show that $\alpha$ is a monomorphism. Let $\beta \colon Y \to X$ be a morphism with $\alpha\beta = 0$. This implies that for every $C \in \mathcal{C}$ the map $\V_0(\beta,[C,E]) \colon \V_0(X,[C,E]) \to \V_0(Y,[C,E])$ is zero. By adjunction this means that $\V_0(\beta \otimes C,E)=0$ and thus $\beta \otimes C=0$ since $E$ is an injective cogenerator. As every object in $\V_0$, in particular the unit object $I$, is a directed colimit of objects from $\mathcal{C}$, and since $\beta \otimes -$ preserves colimits, it follows that $\beta \cong \beta \otimes I = 0$.

By what we have just proved there is a short exact sequence,
\begin{equation*}
  \mathbb{S} \,=\, 
  \xymatrix@C=1.5pc{
  0 \ar[r] & X \ar[r]^-{\alpha} & \textstyle\prod_{C \in \mathcal{C}}\mspace{1mu}[C,E]^{J_C}
  \ar[r] & \Coker\alpha \ar[r] & 0
  }.
\end{equation*}
By construction of $\alpha$, every morphism $X \to [C,E]$ with $C \in \mathcal{C}$ factors through $\alpha$; thus for every $C \in \mathcal{C}$ the morphism $\V_0(\alpha,[C,E])$ is surjective and therefore $\V_0(\mathbb{S},[C,E])$ is a short exact sequence. It follows from \lemref{gps11} that $\mathbb{S}$ is geometrically pure exact. Thus, if $X$ is geometrically pure injective, then $\mathbb{S}$ splits and $X$ is a direct summand in $\prod_{C \in \mathcal{C}}[C,E]^{J_C}$.
\end{proof}

\begin{dfn}[{\cite[Dfn.~2.1]{BQR98}}]
\label{dfn:lambda-small-colimit}
Let $\lambda$ be a regular cardinal (in our applications, $\lambda$ will be the fixed cardinal from \stpref{Setup-cosmos-Grothendieck}) and let $\A$ be a $\V$-category. Let $G \colon \mathcal{K}^{\mathrm{op}} \rightarrow \V$ and $F \colon \mathcal{K} \rightarrow \A$ be $\V$-functors and assume that the weighted colimit $G \star F \in \A$ exists. If the weight $G$ is $\lambda$-small in the sense of \dfnref{lambda-small-functor}, then $G \star F$ is called a \emph{$\lambda$-small weighted colimit}. 

A $\V$-functor $T \colon \A \rightarrow \mathcal{B}$ is said to be \textit{$\lambda$-cocontinuous} if it preserves all $\lambda$-small weighted colimits, that is, for every $\lambda$-small weight $G$ one has $G \star (T \circ F) \cong T(G \star F)$. We set
\begin{equation*}
\lCocon[\lambda]{\A}{\mathcal{B}}
\,=\, 
\big\{\textnormal{the collection of all $\lambda$-cocontinuous $\V$-functors $\A \to \mathcal{B}$} \big\}\;.
\end{equation*} 
\end{dfn}

\begin{rmk}
  \label{rmk:has-a-small-weighted-colimits}
  By \cite[Prop.~3.2 and Cor.~3.3]{BQR98} the full $\V$-subcategory $\A=\lPres{\V}$ is closed under $\lambda$-small weighted colimits in $\V$. Thus $\lPres{\V}$ has all $\lambda$-small weighted colimits. 
\end{rmk}
         
\begin{prp}
  \label{prp:essim}
  For the tensor embedding \textnormal{\eqref{tenemb}} the essential image is:
  \begin{equation*}
    \operatorname{Ess.Im}\upTheta
    \,=\,
    \lCocon{\lPres{\V}}{\V}\;.
  \end{equation*}
\end{prp}

\begin{proof}
  For every object $X \in \V$ the $\V$-functor $\upTheta(X) = (X \otimes -)|_{\lPres{\V}}$ is $\lambda$-cocontinuous. Indeed, $X \otimes - \colon \V \to \V$ preserves all weighted colimits, and by \rmkref{has-a-small-weighted-colimits} any $\lambda$-small weighted colimit in $\lPres{\V}$ is, in fact, a ($\lambda$-small) weighted colimit in $\V$. Thus ``$\subseteq$'' holds.
  
  Conversely, let $T \colon \lPres{\V} \to \V$ be any $\lambda$-cocontinuous $\V$-functor. Consider for any $A \in \lPres{\V}$ the $\V$-functors $G \colon \mathcal{I}^{\mathrm{op}} \to \V$ and $F \colon \mathcal{I} \rightarrow \lPres{\V}$ given by $G(*)=A$ and $F(*)=I$. Here $\mathcal{I}$ is the unit $\V$-category from \ref{ect}.  Note that the weight $G$ is $\lambda$-small as the objects $\mathcal{I}^{\mathrm{op}}(*,*) = I$ and $G(*)=A$ are $\lambda$-presentable. Evidently one has
\begin{equation*}
G \star F \,=\, G(*) \otimes F(*) \,=\, A\otimes I\;,
\end{equation*}
and since $T$ is assumed to preserve $\lambda$-small weighted colimits, it follows that
\begin{equation*}
  T(A) \,\cong\, T(A\otimes I) \,\cong\, T(G \star F) 
  \,\cong\, G \star (T \circ F) \,\cong\, G(*) \otimes (T \circ F)(*) \,\cong\, A \otimes T(I)\;.
\end{equation*}
Hence $T$ is $\V$-naturally isomorphic to $\upTheta(T(I)) = (T(I) \otimes -)|_{\lPres{\V}}$, so $T \in \operatorname{Ess.Im}\upTheta
$.
\end{proof}

\begin{thm}
\label{thm:main2} 
Let $\V$ be as in \stpref{Setup-cosmos-Grothendieck}. The tensor embedding 
\begin{equation*}
\upTheta \colon \V \longrightarrow [\lPres{\V},\V]
\qquad \text{given by} \qquad 
X \longmapsto (X\otimes-)|_{\lPres{\V}}
\end{equation*}
is cocontinuous and it induces two equivalences:
\begin{prt}
\item An equivalence of\, $\V$-categories, $\upTheta \colon \V \stackrel{\simeq}{\longrightarrow} \lCocon{\lPres{\V}}{\V}$.

\item An equivalence of exact categories, $\upTheta_0 \colon (\V_0,\mathscr{E}_\otimes) \stackrel{\simeq}{\longrightarrow} \lCocon{\lPres{\V}}{\V}$, where the latter is an extension-closed subcategory of the abelian category $[\lPres{\V},\V]_0$.
\end{prt}
\end{thm}

\begin{proof}
  Combine \thmref{main}, \prpref{essim}, and \corref{ab-equals-star}.
\end{proof}

\section{The case where $\V$ is generated by dualizable objects}
\label{sec:dualizable}

In this final section, we work with \stpref{Setup-cosmos-Grothendieck-strong} below. We prove in \prpref{aleph0-base}\prtlbl{a} that under the assumptions in \stpref{Setup-cosmos-Grothendieck-strong}, $\V$ is a locally finitely presentable base. Thus in \stpref{Setup-cosmos-Grothendieck} and in all of the results from \secref{Grothendieck-cosmos} we can set $\lambda = \aleph_0$. As it is customary, we write
\begin{equation*}
  \fp{\V} \,:=\, \lPres[\aleph_0]{\V}
\end{equation*}
for the class of finitely presentable objects in $\V$, so the tensor embedding \eqref{tenemb} becomes:
\begin{equation*}
  \upTheta \colon \V \longrightarrow [\fp{\V},\V]\;.
\end{equation*}
We will improve and make \thmref{main2} more explicit in this situation. Let us explain the two main insights we obtain in this section:
\begin{itemlist}
\item We know from \thmref{main2}\prtlbl{b} that the geometrically pure exact category $(\V_0,\mathscr{E}_\otimes)$~is equivalent, as an exact category, to $\operatorname{Ess.Im}\upTheta =
\lCocon[\aleph_0]{\fp{\V}}{\V}$. We prove in \thmref{main3} that $\operatorname{Ess.Im}\upTheta$ also coincides with the class of absolutely pure objects in $[\fp{\V},\V]_0$ in the sense of \dfnref{abspure}.

\item As mentioned in the remarks preceding \lemref{3T}, it follows from \cite[Thm.~4.2]{AG16} that $[\fp{\V},\V]_0$ is a Grothendieck category; in particular, it has enough injectives. \thmref{main3} gives a very concrete description of the injective objects in $[\fp{\V},\V]_0$: they are precisely the $\V$-functors of the form $(X \otimes -)|_{\fp{\V}}$ where $X$ is a geometrically pure injective object in $\V_0$ in the sense of \dfnref{geo-exact-structure}.
\end{itemlist}

\begin{stp}
\label{stp:Setup-cosmos-Grothendieck-strong} 
In this section, $(\V,\otimes,I,[-,-])$ is a cosmos and $\V$ is a Grothendieck category (just as in \stpref{Setup-cosmos-Grothendieck}) subject to the following requirements:
\begin{itemlist}
\item The unit object $I$ is finitely presentable in $\V_0$.
\item The category $\V_0$ is generated by a set of dualizable objects (defined in  \ref{dualizable} below). 
\end{itemlist} 
\end{stp}

\begin{exa}
\label{exa:Ch-new}
The two Grothendieck cosmos from \exaref{Ch}, that is,
\begin{equation}
  \label{eq:Chx2}
  (\Ch{R},\tTen,\stalk{R},\tHom)
  \qquad \textnormal{and} \qquad (\Ch{R},\mTen,\disc{R},\mHom)\;,
\end{equation} 
satisfy \stpref{Setup-cosmos-Grothendieck-strong}. Indeed, it is not hard to see that 
\begin{equation*}
  \fp{\Ch{R}} \,=\, \{ X \in \Ch{R} \,|\, \textnormal{$X$ is bounded and each module $X_n$ is finitely presentable} \}\;,
\end{equation*}
in particular, the unit objects $\stalk{R}$ and $\disc{R}$ of the two cosmos are finitely presentable.~Thus the first condition in \stpref{Setup-cosmos-Grothendieck-strong} holds. Evidently, 
\begin{equation*}
  \{\upSigma^n\disc{R}\,|\, n \in \mathbb{Z}\} \qquad
  \textnormal{(where $\upSigma^n$ is the $n$'th shift of a complex)}
\end{equation*}
is a generating set of objects in $\Ch{R}$. Each object $\upSigma^n\disc{R}$
is dualizable in the rightmost cosmos in \eqref{Chx2} as $I=\disc{R}$ is the unit object. The object $\upSigma^n\disc{R}$
is also dualizable in the leftmost cosmos in \eqref{Chx2}, indeed, it is not hard to see that every \emph{perfect} $R$-complex (i.e.~a bounded complex of finitely generated projective $R$-modules) is dualizable in this cosmos. Hence the second condition in \stpref{Setup-cosmos-Grothendieck-strong} holds as well.
\end{exa}

\begin{exa}
  \label{exa:QcohX-new}
  Consider the Grothendieck cosmos from  \exaref{QcohX}\prtlbl{b}, that is,
  \begin{equation*}
    \label{eq:Qcoh}
    (\Qcoh{X},\otimes_{X},\mathscr{O}_X,\shHomqce_X)\;.
  \end{equation*}  
  If $X$ is \emph{concentrated} (e.g.~if $X$ is quasi-compact
and quasi-separated), then $I=\mathscr{O}_X$ is finitely presentable in $\Qcoh{X}$ by \cite[Prop.~3.7]{EEO16}, so the first condition in \stpref{Setup-cosmos-Grothendieck-strong} holds. 

Note by the way, that for any noetherian scheme $X$, it follows from \cite[Lem.~B.3]{CS17} (and the fact that in a locally noetherian Grothendick category, finitely presentable objects and noetherian objects are the same, see \cite[Chap.~V\S4]{Ste}) that one has:
\begin{equation*}
  \fp{\Qcoh{X}} \,=\, \Coh{X} \,:=\, \{ X \in \Qcoh{X} \,|\, \textnormal{$X$ is coherent} \}\;,
\end{equation*}

   The dualizable objects in $\Qcoh{X}$ are precisely the locally free sheaves of finite rank, see \cite[Prop.~4.7.5]{Bra}, and hence the second condition in \stpref{Setup-cosmos-Grothendieck-strong} holds if and only if $X$ has the \emph{strong resolution property} in the sense of \cite[Dfn.~2.2.7]{Bra}. Many types of schemes---for example, any projective scheme and any separated noetherian locally factorial scheme---do have the strong resolution property; see the remarks after \cite[Dfn.~2.2.7]{Bra}. We refer to \cite[3.4]{MR3720855} for further remarks and insights about the strong resolution property.
\end{exa}

\begin{bfhpg}[Evaluation morphisms] 
\label{eval-morphisms} For any closed symmetric monoidal category there are canonical natural morphisms in $\V_0$ (see e.g.~\cite[III\S1, p.~120]{LewisMay} for construction of the first map):
\begin{equation*}
  \label{eq:eval}
  \vartheta_{XYZ} \colon [X,Y] \otimes Z \longrightarrow [X,Y \otimes Z]
  \qquad \text{and} \qquad
  \eta_{XYZ} \colon X \otimes [Y,Z] \longrightarrow [[X,Y],Z]   
\end{equation*}
In commutative algebra, that is, in the case where $\V=\Mod{R}$ for a commutative ring $R$, people sometimes refer to $\vartheta$ as \emph{tensor evaluation} and to $\eta$ as \emph{homomorphism evaluation}; see for example \cite[(A.2.10) and (A.2.11)]{Chr}. 
\end{bfhpg}

We now consider \emph{dualizable} objects as defined in \cite{HPS} and \cite{LewisMay}. 

\begin{bfhpg}[Dualizable objects] 
\label{dualizable} 
For every object $P$ in $\V$, the following conditions are equivalent:
\begin{eqc}
\item The canonical morphism $[P,I]\otimes P \to [P,P]$ is an isomorphism.
\item The canonical morphism $[P,I]\otimes Z \to [P,Z]$ is an isomorphism for all objects $Z \in \V$.
\item $\vartheta_{PYZ} \colon [P,Y]\otimes Z \to [P,Y\otimes Z]$ is an isomorphism for all objects $Y,Z \in \V$.
\end{eqc}  
Evidently, \eqclbl{iii}\,$\Rightarrow$\,\eqclbl{ii}\,$\Rightarrow$\,\eqclbl{i}. Objects $P$ satisfying \eqclbl{i} are in \cite[III\S1, Dfn.~1.1]{LewisMay} called \emph{finite}, and in Prop.~1.3(ii) in \emph{loc.~cit.} it is proved that \eqclbl{i} implies \eqclbl{iii}. Hence all three conditions are equivalent. Objects $P$ satisfying \eqclbl{ii} are called \emph{(strongly) dualizable} in \cite[Dfn.~1.1.2]{HPS}; we shall adopt this terminology, ``dualizable'', for objects satisfying conditions \eqclbl{i}--\eqclbl{iii}. 
\end{bfhpg}

\begin{rmk}
\label{rmk:P-dual}
If $P \in \V$ is dualizable, then so is $[P,I]$, and the natural map $P \to [[P,I],I]$ is an isomorphism; see \cite[III\S1, Props.~1.2 and 1.3(i)]{LewisMay}. By condition \ref{dualizable}\eqclbl{ii} there is a natural isomorphism $[P,I] \otimes - \cong [P,-]$. By replacing $P$ with $[P,I]$ one also gets $P \otimes - \cong [[P,I],-]$.
\end{rmk}

\begin{lem}
\label{lem:lem-dualizable} 
If $P$ and $Q$ are dualizable, then so are $P \oplus Q$, $P \otimes Q$, and $[P,Q]$.
\end{lem}

\begin{proof}
  Using condition \ref{dualizable}\eqclbl{ii} and additivity of the functors $-\otimes-$ and $[-,-]$, it is easy to see that $P \oplus Q$ is dualizable. By condition \ref{dualizable}\eqclbl{iii} the computation
\begin{equation*}
  [P \otimes Q,Y] \otimes Z
  \,\cong\,
  [P,[Q,Y]] \otimes Z
  \,\cong\,
  [P,[Q,Y] \otimes Z]
  \,\cong\,
  [P,[Q,Y \otimes Z]]
  \,\cong\,
  [P \otimes Q,Y \otimes Z]
\end{equation*}  
shows that $P \otimes Q$ is dualizable. To see that $[P,Q]$ is dualizable, note that $[P,Q] \cong [P,I] \otimes Q$ (as $P$ is dualizable) and apply what we have just proved combined with \rmkref{P-dual}.
\end{proof}

\begin{dfn}
\label{dfn:injective} 
Following Lewis \cite[Dfn.~1.1]{Lew99} we use the next terms for an object $J$ in $\V$.
\begin{align*}
  \textnormal{$J$ is injective} &\ \ \iff \ \
  \textnormal{the functor $\V_0(-,J) \colon \V_0^{\mathrm{op}} \to \Ab$ is exact.}
  \\
  \textnormal{$J$ is \emph{internally injective}} &\ \ \iff \ \
  \textnormal{the functor $[-,J] \colon \V_0^{\mathrm{op}} \to \V_0$ is exact.}  
\end{align*}
\end{dfn}

A question of interest in \cite{Lew99} is when every injective object in $\V$ is internally injective; in this case Lewis would say that $\V$ satisfies the condition {\bf IiII} ({\bf I}njective
{\bf i}mplies {\bf I}nternally {\bf I}njective). As we shall prove next, this condition does hold under \stpref{Setup-cosmos-Grothendieck-strong}.

\begin{prp}
\label{prp:aleph0-base} 
For $\V$ as in \stpref{Setup-cosmos-Grothendieck-strong}, the following assertions hold.
\begin{prt}
\item $\V$ is a locally finitely presentable base (see \ref{lpb}).
\item $\V$ satisfies Lewis' condition {\bf IiII}, i.e.~every injective object is internally injective.
\item For objects $X,Y,J$ in $\V$ where $X$ is finitely presentable and $J$ is injective, the morphism $\eta_{XYJ} \colon X \otimes [Y,J] \longrightarrow [[X,Y],J]$ from \ref{eval-morphisms} is an isomorphism.
\end{prt}
\end{prp}

\begin{proof}
\proofoftag{a} As $I$ is finitely presentable, condition \ref{lpb}\rqmlbl{2} holds with \mbox{$\lambda = \aleph_0$}. Note that every dualizable object $P$ is finitely presentable. Indeed, $[P,-] \colon \V_0 \to \V_0$ preserves colimits as it is naturally isomorphic to $[P,I] \otimes -$ by \rmkref{P-dual}, and $\V_0(I,-)$ preserves directed~colimits as $I$ is finitely presentable. Thus $\V_0(P,-) \cong \V_0(I,[P,-])$ preserves directed colimits. 

As $\V_0$ is a Grothendieck category generated by a set of finitely presentable (even dualizable) objects, it is a locally finitely
presentable Grothendieck category in the sense of Breitsprecher \cite[Dfn.~(1.1)]{Bre70}. Hence $\V_0$ is also a locally finitely
presentable category in the ordinary sense, see \cite[Satz (1.5)]{Bre70} and \cite[(2.4)]{CB}, so condition \ref{lpb}\rqmlbl{1} holds~with~$\lambda = \aleph_0$. 

To see that \ref{lpb}(3) holds, i.e. that the class of finitely presentable objects is closed under $\otimes$, note that as $\V_0$ is generated by a set of dualizable objects, \cite[Satz (1.11)]{Bre70} and \lemref{lem-dualizable} yield that an object $X$ is finitely presentable if and only if there is an exact sequence
\begin{equation}
  \label{eq:PPX}
  P_1 \longrightarrow P_0 \longrightarrow X \longrightarrow 0
\end{equation}
with $P_0,P_1$ dualizable. Now let $Y$ be yet a finitely presentable object and choose an exact sequence $Q_1 \to Q_0 \to Y \to 0$ with $Q_0,Q_1$ dualizable. It is not hard to see that there is an exact sequence $(P_1 \otimes Q_0) \oplus (P_0 \otimes Q_1) \to P_0 \otimes Q_0 \to X\otimes Y \to 0$, cf.~\cite[Chap.~II\S3.6 Prop.~6]{Bou98}, so it follows from \lemref{lem-dualizable} that $X \otimes Y$ is finitely presentable.

\proofoftag{b} As $\V_0$ is generated by a set, say, $\mathcal{P}$ of dualizable objects, the category $\V_0$ has a $\otimes$-flat generator, namely $\bigoplus_{P \in \mathcal{P}} P$. The conclusion now follows from \cite[Lem. 3.1]{SS19}.

\proofoftag{c} As noted in the proof of \prtlbl{a}, if $X$ is a finitely presentable object there exist dualizable objects $P_0$ and $P_1$ and an exact sequence \eqref{PPX}. It induces a commutative diagram in $\V_0$:
\begin{equation}
  \label{eq:eta}
  \begin{gathered}
  \xymatrix{
    P_1 \otimes [Y,J] \ar[r] \ar[d]^-{\eta_{P_1YJ}}
    &
    P_0 \otimes [Y,J] \ar[r] \ar[d]^-{\eta_{P_0YJ}}
    &
    X \otimes [Y,J] \ar[r] \ar[d]^-{\eta_{XYJ}}    
    &
    0
    \\
    [[P_1,Y],J] \ar[r] 
    &
    [[P_0,Y],J] \ar[r] 
    &
    [[X,Y],J] \ar[r] 
    &
    0\;.\mspace{-8mu}    
  }
  \end{gathered}
\end{equation}
In this diagram, the upper row is exact by right exactness of the functor $- \otimes [Y,J]$. The sequence $0 \to [X,Y] \to [P_0,Y] \to [P_1,Y]$ is exact by left exactness of the functor $[-,Y]$. Thus, if $J$ is injective, and hence also internally injective by part \prtlbl{b}, we get exatcness of the lower row in \eqref{eta}. Thus, to prove that $\eta_{XYJ}$ is an isomorphism, it suffices by the Five Lemma to show that $\eta_{P_0YJ}$ and $\eta_{P_1YJ}$ are isomorphisms. However, for every dualizable object $P$ and arbitrary objects $Y,Z$, the map $\eta_{PYZ}$ is an isomorphism. Indeed, \rmkref{P-dual}, the adjunction $(-\otimes Y) \dashv [Y,-]$, and condition \ref{dualizable}\eqclbl{ii} yield isomorphisms:
\begin{equation*}
  P \otimes [Y,Z] \,\cong\, [[P,I],[Y,Z]] \,\cong\, [[P,I] \otimes Y,Z] \,\cong\, [[P,Y],Z]\;.
\end{equation*}  
It is straightforward to verify that the composite of these isomorphisms is $\eta_{PYZ}$. 
\end{proof}

\begin{lem}
\label{lem:product} 
For $\V$ as in \stpref{Setup-cosmos-Grothendieck-strong}, the following assertions hold.
\begin{prt}
\item For any family $\{B_q\}_{q \in Q}$ of objects and every finitely presentable object $A$ in $\V_0$, the next canonical morphism is an isomorphism:
\begin{equation*}
  \big(\textstyle\prod_{q \in Q}\mspace{1mu}[B_q,E]\,\big) \otimes A \stackrel{\cong}{\longrightarrow}
  \textstyle\prod_{q \in Q} \big([B_q,E] \otimes A \big)\;.
\end{equation*}
\item Assume that $\V_0$ satisfies Grothendieck's axiom \textnormal{(AB4*)}, i.e.~the product of a family of epimorphisms is an epimorphism. For any family $\{C_q\}_{q \in Q}$ of objects and every finitely presentable object $A$ in $\V_0$, the next canonical morphism is an isomorphism:
\begin{equation*}
  \big(\textstyle\prod_{q \in Q} C_q \big) \otimes A \stackrel{\cong}{\longrightarrow}
  \textstyle\prod_{q \in Q} \big( C_q \otimes A \big)\;.
\end{equation*}
\end{prt}
\end{lem}

\begin{proof}
 \proofoftag{a} In the computation below, the $1^\mathrm{st}$ and $4^\mathrm{th}$ isomorphisms follow as $[-,E]$ converts coproducts to products; the $2^\mathrm{nd}$ and $5^\mathrm{th}$ isomorphisms follow from \prpref{aleph0-base}\prtlbl{c}; and the $3^\mathrm{rd}$ isomorphism holds as $[A,-]$ preserves directed colimits by \cite[Lem.~2.6 and Cor.~3.3]{BQR98} since $A$ is finitely presentable (as mentioned in the footnote of \stpref{Setup-cosmos-Grothendieck}, we can apply the theory of \cite{BQR98} as $\V$ is a locally presentable base).
  \begin{align*}
    \textstyle\big(\prod_{q \in Q}\, [B_q,E]\,\big) \otimes A
    &\,\cong\,
    \textstyle [\mspace{1mu}\bigoplus_{q \in Q}B_q,E] \otimes A
    \,\cong\,
    \textstyle [\mspace{1mu}[A,\bigoplus_{q \in Q}B_q],E]
    \,\cong\,    
    \textstyle [\mspace{1mu}\bigoplus_{q \in Q}\,[A,B_q],E]
    \\
    &\,\cong\,
    \textstyle \prod_{q \in Q}\,[\mspace{1mu}[A,B_q],E]
    \,\cong\,
    \textstyle\prod_{q \in Q} \big([B_q,E] \otimes A \big)\;.
  \end{align*}
  
  \proofoftag{b} Assume that $\V_0$ satisfies (AB4*). We must show that $-\otimes A$ preserves \textsl{all} products. By the proof of \prpref{aleph0-base}\prtlbl{a} there is an exact sequence $P_1 \to P_0 \to A \to 0$ with $P_0,P_1$ dualizable. In the induced commutative diagram below, the upper row is exact by right exactness of the functor $\big(\textstyle\prod_{q \in Q} C_q \big) \otimes -$, and the lower row is exact as $\V_0$ satisfies~(AB4*).
\begin{equation*}
  \xymatrix{
    \big(\textstyle\prod_{q \in Q} C_q \big) \otimes P_1 \ar[r] \ar[d]^-{\cong}
    &
    \big(\textstyle\prod_{q \in Q} C_q \big) \otimes P_0 \ar[r] \ar[d]^-{\cong}
    &
    \big(\textstyle\prod_{q \in Q} C_q \big) \otimes A \ar[r] \ar[d]
    &
    0
    \\
    \textstyle\prod_{q \in Q} \big( C_q \otimes P_1 \big) \ar[r] 
    &
    \textstyle\prod_{q \in Q} \big( C_q \otimes P_0 \big) \ar[r] 
    &
    \textstyle\prod_{q \in Q} \big( C_q \otimes A \big) \ar[r] 
    &
    0
  }
\end{equation*}  
The two leftmost vertical morphisms are isomorphisms as the functor $-\otimes P_n$ preserves products, indeed, by \rmkref{P-dual} this functor is naturally isomorphic to $[[P_n,I],-]$, which is a right adjoint. By the Five Lemma, the righmost morphism is an isomorphism too.  
\end{proof}

Some important abelian categories fail to satisfy Grothendieck's axiom (AB4*). For instance, this is often the case for the category of quasi-coherent sheaves on a scheme; see \cite[Exa.~4.9]{Krause}. Fortunately, we shall not need the strong conclusion in \lemref{product}\prtlbl{b} (we have only included it for completeness), as the weaker part \prtlbl{a} is sufficient for our purpose (the proof of \prpref{injective} below). We shall also need the next general lemma.

\begin{lem}
\label{lem:functor-summand}
Let $\mathcal{C}$ and $\mathcal{D}$ be additive and idempotent complete categories, $\upPhi \colon \mathcal{C} \to \mathcal{D}$ a fully faithful additive functor, and $C \in \mathcal{C}$ and $D \in \mathcal{D}$ objects. If $D$ is a direct summand in $\upPhi(C)$, then $D$ has the form $D \cong \upPhi(C')$ for some direct summand $C'$ in $C$.
\end{lem}

\begin{proof}
  Since the functor $\upPhi$ is fully faithful and additive, it induces a ring isomorphism,
  \begin{equation}
  \label{eq:End}
    \mathrm{End}_{\mathcal{C}}(C,C) \stackrel{\cong}{\longrightarrow} \mathrm{End}_{\mathcal{D}}(\upPhi(C),\upPhi(C))
    \qquad \textnormal{given by} \qquad f \longmapsto \upPhi(f)\;.
  \end{equation}
  As $D$ is a direct summand in $\upPhi(C)$ there exist morphisms $h \colon D \to \upPhi(C)$ and $k \colon \upPhi(C) \to D$ in $\mathcal{D}$ such that $kh = \mathrm{id}_D$, and thus $hk$ is an idempotent element in  $\mathrm{End}_{\mathcal{D}}(\upPhi(C),\upPhi(C))$. By the ring isomorphism \eqref{End} there is an idempotent $e$ in $\mathrm{End}_{\mathcal{C}}(C,C)$ with $\upPhi(e) = hk$. As $\mathcal{C}$ is idempotent complete there is an object $C' \in \mathcal{C}$ and morphisms $f \colon C' \to C$ and $g \colon C \to C'$ such that $gf = \mathrm{id}_{C'}$ and $fg = e$, in particular, $C'$ is a direct summand in $C$. The morphisms
\begin{equation*}
  \xymatrix{
  \upPhi(C') \ar[r]^-{\upPhi(f)} & \upPhi(C) \ar[r]^-{k} & D
  }
  \qquad \textnormal{and} \qquad
  \xymatrix{
  \upPhi(C') & \upPhi(C) \ar[l]_-{\upPhi(g)} & D \ar[l]_-{h}
  }
\end{equation*}
satisfy $k\upPhi(f) \circ \upPhi(g)h = \mathrm{id}_D$ and $\upPhi(g)h\circ k\upPhi(f) = \mathrm{id}_{\upPhi(C')}$, so $\upPhi(C') \cong D$ as claimed.
\end{proof}

\begin{prp}
\label{prp:injective} 
A $\V$-functor $H \colon \fp{\V} \to \V$ is an injective object in $[\fp{\V},\V]_0$ if and only if it has the form $H \cong \upTheta_0(X)$ for some geometrically pure injective object $X \in \V_0$.
\end{prp}

\begin{proof}  
  ``If'': We start by showing that $\upTheta_0(X)$ is injective in $[\fp{\V},\V]_0$ for every geometrically pure injective object $X$ in $\V$. By \prpref{gpi-product} (and additivity of the functor $\upTheta_0$), we may assume that $X=\prod_{q \in Q}[B_q,E]$ for some family $\{B_q\}_{q \in Q}$ of finitely presentable objects. It follows from \lemref{product}\prtlbl{a} (notice that this isomorphism is $\V$-natural in $A$) that:
\begin{equation*}
  \textstyle
  \upTheta_0(X) \,=\, \upTheta_0\big( \prod_{q \in Q}[B_q,E] \big)
  \,\cong\, 
  \prod_{q \in Q}\upTheta_0([B_q,E])\;.
\end{equation*} 
Thus we may further reduce the case where $X=[B,E]$ for a single finitely presentable~ob\-ject $B$. Now \prpref{aleph0-base}\prtlbl{c} yields the isomorphism in the following computation:
\begin{equation*}
    \upTheta_0(X) \,=\, ([B,E]\otimes-)\big|_{\fp{\V}}
    \,\cong\, [\mspace{1mu}[-,B],E]\mspace{1.5mu}\big|_{\fp{\V}}
    \,=\,
    [\fp{\V}(-,B),E]\;.
\end{equation*}
The latter in an injective object in $[\fp{\V},\V]_0$ by \lemref{c1}\prtlbl{b} with $\mathcal{K} = \fp{\V}$.

``Only if'': By \lemref{c1} with $\mathcal{K} = \fp{\V}$ and $\mathcal{S} = \{E\}$ we get that $\{\mspace{1mu}[\mspace{1mu}[-,B],E]\mspace{1mu}\}_{B\in \fp{\V}}$ is a cogenerating set of (injective) objects in $[\fp{\V},\V]_0$. So every $H$ in $[\fp{\V},\V]_0$ can be embedded into a product $F = \prod_{q \in Q}[\mspace{1mu}[-,B_q],E]$ where each $B_q$ is finitely presentable. Set 
\begin{equation*}
X\,=\,\textstyle\prod_{q \in Q}[B_q,E] \in \V\;;
\end{equation*}
this is a geometrically pure injective object by \lemref{dual}, and the arguments above show that $F \cong \upTheta_0(X)$. Thus we have an embedding $H \rightarrowtail \upTheta_0(X)$. Consequently, if $H$ is injective, then it is a direct summand in $\upTheta_0(X)$, and it follows from \lemref{functor-summand} that $H \cong \upTheta_0(X')$ for some direct summand $X'$ in $X$. As $X$ is geometrically pure injective, so is $X'$.
\end{proof}

We can now give the result that is explained in the beginning of the section.

\begin{thm}
\label{thm:main3}
Let $\V$ be as in \stpref{Setup-cosmos-Grothendieck-strong}. The underlying tensor embedding functor,
\begin{equation*}
\upTheta_0 \colon \V_0 \longrightarrow [\fp{\V},\V]_0
\qquad \text{given by} \qquad 
X \longmapsto (X\otimes-)|_{\fp{\V}}\;,
\end{equation*}
induces a commutative diagram of exact categories and exact functors,
\begin{equation}
  \label{eq:main3}
  \begin{gathered}
  \xymatrix@C=3pc{
    \big(\V_0,\mathscr{E}_\otimes\big) \ar[r]^-{\upTheta_0}_-{\simeq} & 
    \big(\AbsPure{[\fp{\V},\V]_0},\mathscr{E}_\mathrm{ab}\big)
    \\
    \big(\PureInj{\V_0},\mathscr{E}_\mathrm{split}\big) 
    \ar[u]^-{\mathrm{inc}} \ar[r]^-{\upTheta_0}_-{\simeq} & 
    \big(\Inj{[\fp{\V},\V]_0},\mathscr{E}_\mathrm{split}\big)\;,\mspace{-8mu} \ar[u]_-{\mathrm{inc}}
  }
  \end{gathered}
\end{equation}
where $\mathscr{E}_\otimes$ is the geometrically pure exact structure (\dfnref{geo-exact-structure}), $\mathscr{E}_\mathrm{ab}$ denotes the exact structure on $\AbsPure{[\fp{\V},\V]_0}$ induced by the abelian structure on $[\fp{\V},\V]_0$, and $\mathscr{E}_\mathrm{split}$ is the (trivial) split exact structure. Further, \textnormal{``$\mathrm{inc}$''} denotes the inclusion functor.

In this diagram, the vertical functors are equivalences of exact categories.
\end{thm}

\begin{proof}
  The asserted equivalence of exact categories in the top of diagram \eqref{main3} will follow from \thmref{main}\prtlbl{b} and \corref{ab-equals-star} once we have proved the equality:
  \begin{equation}
    \label{eq:EssImAbsPure}
    \operatorname{Ess.Im}\upTheta_0 \,=\, \AbsPure{[\fp{\V},\V]_0}\;.
  \end{equation} 
  It then follows from \prpref{injective} that the equivalence in the top of \eqref{main3} restrics to the one in the bottom. We now prove the equality \eqref{EssImAbsPure}.
  
  ``$\subseteq$'': Let $H$ be in $\operatorname{Ess.Im}\upTheta_0$, that is,  $H \cong \upTheta_0(X)$ for some $X \in \V$. We must argue that every short exact sequence $0 \to \upTheta_0(X) \to F \to F' \to 0$ in $[\fp{\V},\V]_0$ is $\star$-pure exact. By \lemref{3T} it suffices to prove that $\upTheta_0(X) \rightarrowtail F$ is a \emph{$\star$-pure monomorphism}, meaning~that $G \star \upTheta_0(X) \to G \star F$ is monic for every $\V$-functor $G \colon \fp{\V}^\mathrm{op} \to \V$. By \prpref{dual2} there is a geometrically pure injective object $J$ and a geometrically pure monomorphism $X \rightarrowtail J$ in $\V_0$. Thus $\upTheta_0(X) \rightarrowtail \upTheta_0(J)$ is a $\star$-pure monomorphism by \lemref{i-ii}. As $\upTheta_0(J)$ is injective in $[\fp{\V},\V]_0$ by \prpref{injective}, the morphism $\upTheta_0(X) \to \upTheta_0(J)$ admits a lift:
\begin{equation*}
  \xymatrix{
  \upTheta_0(X) \ar[d] \ar@{>->}[r] & F \ar@{.>}[dl]
  \\
  \upTheta_0(J) & {}
  }
\end{equation*}  
As $\upTheta_0(X) \rightarrowtail \upTheta_0(J)$ is a $\star$-pure monomorphism, so is $\upTheta_0(X) \rightarrowtail F$.

``$\supseteq$'': Let $H$ be an absolutely pure object in $[\fp{\V},\V]_0$. As this category has enough injectives, it follows from \prpref{injective} that there exists a monomorphism $H \rightarrowtail \upTheta_0(J)$ for some geometrically pure injective object $J \in \V$. By assumption on $H$, this is even a $\star$-pure monomorphism, and thus $[H(-),E]$ is a direct summand in $[\upTheta_0(J)(-),E] \cong [-,[J,E]]$ by the equivalent conditions in \prpref{star}. By \cite[\S2.4]{Kelly} the Yoneda embedding
\begin{equation*}
  \upUpsilon_0 \colon \V_0 \longrightarrow [\fp{\V}^\mathrm{op},\V]_0
  \qquad \text{given by} \qquad 
  X \longmapsto [-,X]\big|_{\fp{V}}
\end{equation*}
is fully faithful, so because $[H(-),E]$ is a direct summand in $\upUpsilon_0([J,E]) = [-,[J,E]]$, it follows from \lemref{functor-summand} that $[H(-),E] \cong [-,Y]$ for some direct summand $Y$ in $[J,E]$. By evaluating this isomorphism on the unit object, it follows that $Y \cong [I,Y] \cong [H(I),E]$, so
\begin{equation*}
  [H(-),E] \,\cong\, [-,[H(I),E]] \,\cong\, [H(I) \otimes -,E]\;.
\end{equation*}
It can be verified that this $\V$-natural isomorphism is $[\theta_H,E]$ where
\begin{equation*}
  \theta_H \colon \upTheta_0(H(I)) = H(I) \otimes - \,\longrightarrow\, H(-)
\end{equation*}
is the $\V$-natural transformation from the proof of \lemref{induced-exact}.By \lemref{ev1} 
the functor $[-,E]$ is faithful, and hence it reflects isomorphisms. We conclude that $\theta_H$ is a $\V$-natural isomorphism, and hence $H$ belongs to $\operatorname{Ess.Im}\upTheta_0$.
\end{proof}

We end this paper with a follow-up of the discussion at the end of the Introduction.
  
\begin{rmk}
  \label{rmk:monoid}   
    Let $(\V,\otimes,I,[-,-])$ be a closed symmetric moniodal category and let $R$ be a \emph{monoid} in $\V$. Write $R\textnormal{-Mod}$ (respectively, $\textnormal{Mod-}R$) for the category of \emph{$R$-left-objects} (respectively, \emph{$R$-right-objects}) in $\V$; see \cite[\S2]{Par77}. Note that $R\textnormal{-Mod}$ and $\textnormal{Mod-}R$ are complete, cocomplete, abelian, or Grothendieck if $\V$ is so. Moreover, there are functors
\begin{align}
  \nonumber
  {}_R[-,-] \colon (R\textnormal{-Mod})^\mathrm{op} \times R\textnormal{-Mod} 
  &\longrightarrow \V\;,
  \\
  \label{eq:three-functors}
  [-,-]_R \colon (\textnormal{Mod-}R)^\mathrm{op} \times \textnormal{Mod-}R 
  &\longrightarrow \V\;, \quad \textnormal{and} 
  \\
  \nonumber
  - \otimes_R - \colon \textnormal{Mod-}R \times R\textnormal{-Mod} 
  &\longrightarrow \V\;,
\end{align}    
which behave in expected ways. For example, for $X \in \textnormal{Mod-}R$ the functor $[X,-]$ (which a priori is a functor from $\V$ to $\V$) takes values in $R\textnormal{-Mod}$ and yields a right adjoint of $X \otimes_R -$.

For any full subcategory $\A$ of $R\textnormal{-Mod}$ one can now consider the tensor embedding
\begin{equation*}
\upTheta_0 \colon \textnormal{Mod-}R \longrightarrow [\A,\V]_0
\qquad \text{given by} \qquad 
X \longmapsto (X\otimes_R-)|_{\A}\;.
\end{equation*}
With these functors at hand, we leave it to the reader to formulate appropriate versions of, for example, Theorems A--D from the Introduction and check how the existing proofs can be modified to show these. Concerning Theorem C one can use the adjunctions associated with the functors in \eqref{three-functors} to show that if $\V$ is locally $\lambda$-presentable, then so is $\textnormal{Mod-}R$. To prove that geometrically pure injective objects in $\textnormal{Mod-}R$ correspond to injective objects in $[\fp{R\textnormal{-Mod}},\V]_0$ (as in Theorem~D), a crucial input is the hypothesis that for all objects $X,Y \in \textnormal{Mod-}R$, where $X$ is finitely presentable, and $J \in \V$ is injective, the following canonical morphism is an isomorphism,
\begin{equation*}
  X \otimes_R [Y, J] \longrightarrow [[X,Y]_R, J]\;.
\end{equation*}

  We briefly mention a few examples. A monoid in $\Ab$ is nothing but a ring. For any ring $R$, the stalk complex $\stalk{R}$ and the disc complex $\disc{R}$ from \exaref{Ch} are monoids in $(\Ch{\mathbb{Z}},\tTen[\mathbb{Z}])$ and in $(\Ch{\mathbb{Z}},\mTen[\mathbb{Z}])$, respectively. This viewpoint allows one to deal with \exaref{Ch} also in the case where $R$ is non-commutative.
\end{rmk}

\def\cprime{$'$}
  \providecommand{\arxiv}[2][AC]{\mbox{\href{http://arxiv.org/abs/#2}{\tt
  arXiv:#2 [math.#1]}}}
  \providecommand{\oldarxiv}[2][AC]{\mbox{\href{http://arxiv.org/abs/math/#2}{\sf
  arXiv:math/#2
  [math.#1]}}}\providecommand{\MR}[1]{\mbox{\href{http://www.ams.org/mathscinet-getitem?mr=#1}{#1}}}
  \renewcommand{\MR}[1]{\mbox{\href{http://www.ams.org/mathscinet-getitem?mr=#1}{#1}}}
\providecommand{\bysame}{\leavevmode\hbox to3em{\hrulefill}\thinspace}
\providecommand{\MR}{\relax\ifhmode\unskip\space\fi MR }
\providecommand{\MRhref}[2]{%
  \href{http://www.ams.org/mathscinet-getitem?mr=#1}{#2}
}
\providecommand{\href}[2]{#2}

%\bibliographystyle{amsplain}
%\bibliography{+references}

\begin{thebibliography}{10}

\bibitem{AR}
Ji\v{r}\'{\i} Ad\'{a}mek and Ji\v{r}\'{\i} Rosick\'{y}, \emph{Locally
  presentable and accessible categories}, London Math. Soc. Lecture Note Ser.,
  vol. 189, Cambridge University Press, Cambridge, 1994. \MR{MR1294136}

\bibitem{AR04}
\bysame, \emph{On pure quotients and pure subobjects}, Czechoslovak Math. J.
  \textbf{54(129)} (2004), no.~3, 623--636. \MR{MR2086721}

\bibitem{AG16}
Hassan Al~Hwaeer and Grigory Garkusha, \emph{Grothendieck categories of
  enriched functors}, J. Algebra \textbf{450} (2016), 204--241. \MR{MR3449691}

\bibitem{AJPV08}
Leovigildo Alonso~Tarr\'{\i}o, Ana Jerem\'{\i}as~L\'{o}pez, Marta
  P\'{e}rez~Rodr\'{\i}guez, and Mar\'{\i}a~J. Vale~Gonsalves, \emph{The derived
  category of quasi-coherent sheaves and axiomatic stable homotopy}, Adv. Math.
  \textbf{218} (2008), no.~4, 1224--1252. \MR{MR2419383}

\bibitem{MR3720855}
\bysame, \emph{On the derived category of quasi-coherent sheaves on an {A}dams
  geometric stack}, J. Pure Appl. Algebra \textbf{222} (2018), no.~4, 828--845.
  \MR{MR3720855}

\bibitem{BCE19}
Silvana Bazzoni, Manuel Cort\'{e}s-Izurdiaga, and Sergio Estrada,
  \emph{Periodic modules and acyclic complexes}, Algebr. Represent. Theory (to
  appear), published online 7 August 2019, 23 pp.,
  \url{https://doi.org/10.1007/s10468-019-09918-z}.

\bibitem{Beke}
Tibor Beke, \emph{Sheafifiable homotopy model categories}, Math. Proc.
  Cambridge Philos. Soc. \textbf{129} (2000), no.~3, 447--475. \MR{MR1780498}

\bibitem{B2}
Francis Borceux, \emph{Handbook of categorical algebra 2. {C}ategories and
  structures.}, Encyclopedia Math. Appl., vol.~51, Cambridge University Press,
  Cambridge, 1994. \MR{MR1313497}

\bibitem{BQR98}
Francis Borceux, Carmen Quinteiro~Sandomingo, and Ji\v{r}\'{\i} Rosick\'{y},
  \emph{A theory of enriched sketches}, Theory Appl. Categ. \textbf{4} (1998),
  no.~3, 47--72. \MR{MR1624638}

\bibitem{Bou98}
Nicolas Bourbaki, \emph{Algebra. {I}. {C}hapters 1--3}, Elements of Mathematics
  (Berlin), Springer-Verlag, Berlin, 1989, Translated from the French, Reprint
  of the 1974 edition. \MR{MR979982}

\bibitem{Bra}
Martin Brandenburg, \emph{Tensor categorical foundations of algebraic
  geometry}, PhD thesis, 247 pp., 2014, \arxiv[AG]{1410.1716v1}.

\bibitem{Bre70}
Siegfried Breitsprecher, \emph{Lokal endlich pr\"{a}sentierbare
  {G}rothendieck-{K}ategorien}, Mitt. Math. Sem. Giessen \textbf{85} (1970),
  1--25. \MR{MR262330}

\bibitem{Buhler}
Theo B\"{u}hler, \emph{Exact categories}, Expo. Math. \textbf{28} (2010),
  no.~1, 1--69. \MR{MR2606234}

\bibitem{Chr}
Lars~Winther Christensen, \emph{Gorenstein dimensions}, Lecture Notes in Math.,
  vol. 1747, Springer-Verlag, Berlin, 2000. \MR{MR1799866}

\bibitem{CS17}
Pavel Coupek and Jan \v{S}\v{t}ov\'{\i}\v{c}ek, \emph{Cotilting sheaves on
  noetherian schemes}, preprint, 39 pp., 2019, \arxiv[AG]{1707.01677v2}.

\bibitem{CB}
William Crawley-Boevey, \emph{Locally finitely presented additive categories},
  Comm. Algebra \textbf{22} (1994), no.~5, 1641--1674. \MR{MR1264733}

\bibitem{EEO16}
Edgar~E. Enochs, Sergio Estrada, and Sinem Odaba\c{s}{\i}, \emph{Pure injective
  and absolutely pure sheaves}, Proc. Edinb. Math. Soc. (2) \textbf{59} (2016),
  no.~3, 623--640. \MR{MR3572762}

\bibitem{EG97}
Edgar~E. Enochs and Juan~R. Garc\'{\i}a~Rozas, \emph{Tensor products of
  complexes}, Math. J. Okayama Univ. \textbf{39} (1997), 17--39 (1999).
  \MR{MR1680739}

\bibitem{rha}
Edgar~E. Enochs and Overtoun M.~G. Jenda, \emph{Relative homological algebra},
  de Gruyter Exp. Math., vol.~30, Walter de Gruyter \& Co., Berlin, 2000.
  \MR{MR1753146}

\bibitem{EGO17}
Sergio Estrada, James Gillespie, and Sinem Odaba\c{s}{\i}, \emph{Pure exact
  structures and the pure derived category of a scheme}, Math. Proc. Cambridge
  Philos. Soc. \textbf{163} (2017), no.~2, 251--264. \MR{MR3682629}

\bibitem{Fox}
Thomas~F. Fox, \emph{Purity in locally-presentable monoidal categories}, J.
  Pure Appl. Algebra \textbf{8} (1976), no.~3, 261--265. \MR{MR409597}

\bibitem{Freyd}
Peter Freyd, \emph{Abelian categories. {A}n introduction to the theory of
  functors}, Harper's Series in Modern Mathematics, Harper \& Row, Publishers,
  New York, 1964. \MR{MR0166240}

\bibitem{MR1693036}
Juan~R. Garc\'{\i}a~Rozas, \emph{Covers and envelopes in the category of
  complexes of modules}, Chapman \& Hall/CRC Res. Notes Math., vol. 407,
  Chapman \& Hall/CRC, Boca Raton, FL, 1999. \MR{MR1693036}

\bibitem{AGr57}
Alexander Grothendieck, \emph{Sur quelques points d'alg\`ebre homologique},
  T\^ohoku Math. J. (2) \textbf{9} (1957), 119--221. \MR{MR0102537}

\bibitem{MR3075000}
Alexander Grothendieck and Jean~A. Dieudonn\'e, \emph{El\'ements de
  {G}\'eom\'etrie {A}lg\'ebrique. {I}}, Grundlehren Math. Wiss. [Fundamental
  Principles of Mathematical Sciences], vol. 166, Springer-Verlag, Berlin,
  1971. \MR{MR3075000}

\bibitem{Hartshorne}
Robin Hartshorne, \emph{Algebraic geometry}, Grad. Texts in Math., vol.~52,
  Springer-Verlag, New York-Heidelberg, 1977. \MR{MR0463157}

\bibitem{HPS}
Mark Hovey, John~H. Palmieri, and Neil~P. Strickland, \emph{Axiomatic stable
  homotopy theory}, Mem. Amer. Math. Soc. \textbf{128} (1997), no.~610, x+114.
  \MR{MR1388895}

\bibitem{JL}
Christian~U. Jensen and Helmut Lenzing, \emph{Model-theoretic algebra with
  particular emphasis on fields, rings, modules}, Algebra Logic Appl., vol.~2,
  Gordon and Breach Science Publishers, New York, 1989. \MR{MR1057608}

\bibitem{KS06}
Masaki Kashiwara and Pierre Schapira, \emph{Categories and sheaves},
  Grundlehren Math. Wiss. [Fundamental Principles of Mathematical Sciences],
  vol. 332, Springer-Verlag, Berlin, 2006. \MR{MR2182076}

\bibitem{Kelly}
Gregory~M. Kelly, \emph{Basic concepts of enriched category theory}, Repr.
  Theory Appl. Categ. (2005), no.~10, vi+137, Reprint of the 1982 original
  [Cambridge Univ. Press, Cambridge; MR0651714]. \MR{MR2177301}

\bibitem{Krause}
Henning Krause, \emph{The stable derived category of a {N}oetherian scheme},
  Compos. Math. \textbf{141} (2005), no.~5, 1128--1162. \MR{MR2157133}

\bibitem{Lenzing}
Helmut Lenzing, \emph{Homological transfer from finitely presented to infinite
  modules}, Abelian group theory ({H}onolulu, {H}awaii, 1983), Lecture Notes in
  Math., vol. 1006, Springer, Berlin, 1983, pp.~734--761. \MR{MR722664}

\bibitem{Lew99}
L.~Gaunce Lewis, Jr., \emph{When projective does not imply flat, and other
  homological anomalies}, Theory Appl. Categ. \textbf{5} (1999), no.~9,
  202--250. \MR{MR1711565}

\bibitem{LewisMay}
L.~Gaunce Lewis, Jr., J.~Peter May, Mark Steinberger, and James~E. McClure,
  \emph{Equivariant stable homotopy theory}, Lecture Notes in Math., vol. 1213,
  Springer-Verlag, Berlin, 1986, With contributions by James E. McClure.
  \MR{MR866482}

\bibitem{MacLane63}
Saunders Mac~Lane, \emph{Natural associativity and commutativity}, Rice Univ.
  Studies \textbf{49} (1963), no.~4, 28--46. \MR{MR0170925}

\bibitem{MacLane}
\bysame, \emph{Categories for the working mathematician}, second ed., Grad.
  Texts in Math., vol.~5, Springer-Verlag, New York, 1998. \MR{MR1712872}

\bibitem{MP}
Michael Makkai and Robert Par\'{e}, \emph{Accessible categories: the
  foundations of categorical model theory}, Contemp. Math., vol. 104, Amer.
  Math. Soc., Providence, RI, 1989. \MR{MR1031717}


\bibitem{Par77} Bodo Pareigis, \emph{Non-additive ring and module theory. I: General theory of monoids},  Publ. Math. Debrecen \textbf{24} (1977), no.~1--2, 189--204. \MR{MR0450361}
  
\bibitem{Quillen}
Daniel Quillen, \emph{Higher algebraic {$K$}-theory. {I}}, Algebraic
  {$K$}-theory, {I}: {H}igher {$K$}-theories ({P}roc. {C}onf., {B}attelle
  {M}emorial {I}nst., {S}eattle, {W}ash., 1972), Lecture Notes in Math., vol.
  341, 1973, pp.~85--147. \MR{MR0338129}

\bibitem{SS19}
Alexander Sl{\'a}vik and Jan \v{S}\v{t}ov\'{\i}\v{c}ek, \emph{On flat
  generators and {M}atlis duality for quasicoherent sheaves}, preprint, 8 pp.,
  2019, \arxiv[AG]{1902.05740v2}.

\bibitem{stacks-project}
The {Stacks Project Authors}, \emph{\itshape stacks project},
  \url{http://stacks.math.columbia.edu}, 2017.

\bibitem{Ste}
Bo~Stenstr\"{o}m, \emph{Rings of quotients. an introduction to methods of ring
  theory}, Grundlehren Math. Wiss. [Fundamental Principles of Mathematical
  Sciences], vol. 217, Springer-Verlag, New York-Heidelberg, 1975.
  \MR{MR0389953}

\bibitem{Sto14}
Jan {\v{S}}\v{t}ov\'{\i}\v{c}ek, \emph{On purity and applications to coderived
  and singularity categories}, preprint, 45 pp., 2014, \arxiv[CT]{1412.1615v1}.

\bibitem{TT90}
Robert~W. Thomason and Thomas Trobaugh, \emph{Higher algebraic {$K$}-theory of
  schemes and of derived categories}, The {G}rothendieck {F}estschrift, {V}ol.
  {III}, Progr. Math., vol.~88, Birkh\"{a}user Boston, Boston, MA, 1990,
  pp.~247--435. \MR{MR1106918}

\end{thebibliography}

\end{document}